\documentclass{article}

\usepackage{amsmath,amssymb,bbm}
\usepackage{color}
\usepackage{mathtools}

\DeclareMathOperator*{\R}{\mathbb{R}}
\DeclareMathOperator*{\E}{\mathbb{E}}

\DeclareMathOperator*{\Rd}{\mathbb{R}^d}

\newcommand{\vz}{\mathbf{0}}

\usepackage{microtype}
\usepackage{graphicx}
\usepackage{subfigure}
\usepackage{booktabs} 
\graphicspath{{image/}}

\usepackage{url}

\usepackage{hyperref}

\usepackage{algorithm, algorithmic}
\makeatletter
\newcounter{procedure}
\newenvironment{procedure}[1][htb]{%
    \let\c@algorithm\c@procedure
    \renewcommand{\ALG@name}{Procedure}
   \begin{algorithm}[#1]%
  }{\end{algorithm}}
\makeatother

\usepackage[accepted]{icml2019}

\usepackage{etoolbox}
\makeatletter
\patchcmd\@combinedblfloats{\box\@outputbox}{\unvbox\@outputbox}{}{%
   \errmessage{\noexpand\@combinedblfloats could not be patched}%
}%
 \makeatother

\icmltitlerunning{Acceleration of SVRG and Katyusha X by Inexact Preconditioning}

\usepackage{mathtools}
\usepackage[normalem]{ulem} 

\usepackage{amsmath}
\usepackage{amssymb}
\usepackage{mathtools}
\usepackage{mathrsfs}

\newcommand{\cut}[1]{{}}

\newcommand{\RR}{\mathbb{R}}				

\newcommand{\dom}[1]{{\mathrm{dom}(#1)}} 	

\newcommand{\prox}{\mathbf{prox}}

\DeclareMathOperator*{\argmin}{arg\,min}

\newcommand{\DeclareAutoPairedDelimiter}[3]{%
	\expandafter\DeclarePairedDelimiter\csname Auto\string#1\endcsname{#2}{#3}%
	\begingroup\edef\x{\endgroup
		\noexpand\DeclareRobustCommand{\noexpand#1}{%
			\expandafter\noexpand\csname Auto\string#1\endcsname*}}%
	\x}

\DeclareAutoPairedDelimiter{\p}{(}{)} 					
\DeclareAutoPairedDelimiter{\sp}{[}{]} 					
\DeclareAutoPairedDelimiter{\abs}{|}{|} 					
\DeclareAutoPairedDelimiter{\cp}{\{}{\}} 				
\DeclareAutoPairedDelimiter{\dotp}{\langle}{\rangle} 	
\DeclareAutoPairedDelimiter{\n}{\Vert}{\Vert} 			
\DeclareAutoPairedDelimiter{\cl}{\lceil}{\rceil}

\newcommand{\cO}{{\mathcal{O}}}

\usepackage{amsmath}
\usepackage{amsthm}
\usepackage{amssymb}
\usepackage{mathtools}
\usepackage{mathrsfs}

\usepackage[nameinlink]{cleveref}

\newcommand{\bc}{\begin{center}}
\newcommand{\ec}{\end{center}}

\newcommand{\bdm}{\begin{displaymath}}
\newcommand{\edm}{\end{displaymath}}

\newcommand{\beq}{\begin{equation}}
\newcommand{\eeq}{\end{equation}}

\newcommand{\bfl}{\begin{flushleft}}
\newcommand{\efl}{\end{flushleft}}

\newcommand{\bt}{\begin{tabbing}}
\newcommand{\et}{\end{tabbing}}

\newcommand{\beqn}{\begin{align}}
\newcommand{\eeqn}{\end{align}}

\newcommand{\beqs}{\begin{align*}} 
\newcommand{\eeqs}{\end{align*}}  

\newtheoremstyle{Fancyplain}
{\topsep}   
{\topsep}   
{\itshape}  
{0pt}       
{\bfseries} 
{}         
{5pt plus 1pt minus 1pt} 
{\thmname{#1} \thmnumber{#2}. \thmnote{\normalfont\bfseries#3.}}

\theoremstyle{Fancyplain}
\newtheorem{thm}{Theorem}
\newtheorem{lem}{Lemma}

\crefname{thm}{Thm.}{Thms.}
\Crefname{thm}{Theorem}{Theorems}
\crefname{lem}{Lem.}{Lems.}
\Crefname{lem}{Lemma}{Lemmas}
\crefname{cor}{Cor.}{Cors.}
\Crefname{cor}{Corollary}{Corollaries}
\crefname{prop}{Prop.}{Props.}
\Crefname{prop}{Proposition}{Propositions}

\newtheoremstyle{Fancydefinition}
{\topsep}   
{\topsep}   
{\normalfont}  
{0pt}       
{\bfseries} 
{}         
{5pt plus 1pt minus 1pt} 
{\thmname{#1} \thmnumber{#2}. \thmnote{\normalfont\bfseries#3.}}

\theoremstyle{Fancydefinition}
\newtheorem{defn}{Definition}

\newtheorem{rem}{Remark}
\newtheorem{asmp}{Assumption}

\crefname{defn}{Defn.}{Defns.}
\Crefname{defn}{Definition}{Definitions}
\crefname{example}{Ex.}{Exs.}
\Crefname{example}{Example}{Examples}
\crefname{xca}{Ex.}{Exs.}
\Crefname{xca}{Exercise}{Exercises}
\crefname{rem}{Rem.}{Rems.}
\Crefname{rem}{Remark}{Remarks}
\crefname{asmp}{Asmp.}{Asmps.}
\Crefname{asmp}{Assumption}{Assumptions}
\crefname{section}{Sec.}{Secs.}
\Crefname{section}{Section}{Sections}

\numberwithin{equation}{section}
\numberwithin{figure}{section}

\begin{document}

\twocolumn[
\icmltitle{Acceleration of SVRG and Katyusha X by Inexact Preconditioning}




\begin{icmlauthorlist}
\icmlauthor{Yanli Liu}{LA}
\icmlauthor{Fei Feng}{LA}
\icmlauthor{Wotao Yin}{LA}
\end{icmlauthorlist}

\icmlaffiliation{LA}{Department of Mathematics, University of California, Los Angeles, Los Angeles, CA, USA}
\icmlaffiliation{LA}{Department of Mathematics, University of California, Los Angeles, Los Angeles, CA, USA}
\icmlaffiliation{LA}{Department of Mathematics, University of California, Los Angeles, Los Angeles, CA, USA}

\icmlcorrespondingauthor{Yanli Liu}{yanli@math.ucla.edu}

\icmlkeywords{SVRG, Katyusha X, inexact preconditioning}

\vskip 0.3in
]



\printAffiliationsAndNotice{} 

\begin{abstract}
Empirical risk minimization is an important class of optimization problems with many popular machine learning applications, and stochastic variance reduction methods are popular choices for solving them. Among these methods, SVRG and Katyusha X (a Nesterov accelerated SVRG) achieve fast convergence without substantial memory requirement. In this paper, we propose to accelerate these two algorithms by \textit{inexact preconditioning}, the proposed methods employ \textit{fixed} preconditioners, although the subproblem in each epoch becomes harder, it suffices to apply \textit{fixed} number of simple subroutines to solve it inexactly, without losing the overall convergence. As a result, this inexact preconditioning strategy gives provably better iteration complexity and gradient complexity over SVRG and Katyusha X. We also allow each function in the finite sum to be nonconvex while the sum is strongly convex. In our numerical experiments, we observe an on average $8\times$ speedup on the number of iterations and $7\times$ speedup on runtime.

\end{abstract}

\section{Introduction}
\label{sec: introduction}
Empirical risk minimization is an important class of optimization problems that has many applications in machine learning, especially in the large-scale setting. In this paper, we formulate it as the minimization of the following objective
\begin{align}
\label{equ: objective}
F(x)= f(x)+\psi(x)=\frac{1}{n}\sum_{i=1}^n f_i(x)+\psi(x),
\end{align}
where the finite sum $f(x)$ is strongly convex, each $f_i(x)$ in the finite sum is smooth\footnotemark[1] and can be nonconvex, and the regularizer $\psi(x)$ is proper, closed, and convex, but may be nonsmooth. A nonzero $\psi(x)$ is desirable in many applications, for example, $\ell_1-$ regularization that induces sparsity in the solution. Allowing $f_i$ to be nonconvex is also necessary in some applications, e.g.,  shift-and-invert approach to solve PCA \cite{saad1992numerical}. 
\footnotetext[1]{A function $f$ is said to be smooth if its gradient $\nabla f$ is Lipschitz continuous.}
\subsection{Related Work}
To obtain a high quality approximate solution $\hat{x}$ of \eqref{equ: objective}, 
stochastic variance reduction algorithms are a class of preferable choices in the large scale setting where $n$ is huge. If each $f_i$ is $\sigma-$strongly convex and $L-$smooth, and $\psi=0$, then SVRG \cite{JohnsonZhang2013_accelerating}, SAGA \cite{DefazioBachLacoste-Julien2014_saga}, SAG \cite{RouxSchmidtBach2012_stochastic}, SARAH \cite{nguyen2017sarah}, SDCA \cite{Shalev-ShwartzZhang2013_stochastic}, SDCA  without duality \cite{shalev2016sdca}, and Finito/MISO \cite{DefazioDomkeCaetano2014_finito, Mairal2013_optimization} can find such a $\hat{x}$ within $\cO\big((n+\frac{L}{\sigma})\ln(\frac{1}{\varepsilon})\big)$ evaluations of component gradients $\nabla f_i$, while vanilla gradient descent needs $\cO(n\frac{L}{\sigma}\ln{\frac{1}{\varepsilon}})$ evaluations. Recently, SCSG improves this complexity to $\cO\big((n\wedge\frac{L}{\sigma\varepsilon}+\frac{L}{\sigma})\ln\frac{1}{\varepsilon}\big)$\footnotemark[2]. When $\psi\neq 0$, many of these algorithms can be extended accordingly and the same gradient complexity is preserved \cite{xiao2014proximal,DefazioBachLacoste-Julien2014_saga,Shalev-ShwartzZhang2016_accelerated}. Among these methods, SVRG has been a popular choice due to its low memory cost.
\footnotetext[2]{$a\wedge b\coloneqq \min\{a, b\}$.}

When the condition number $\frac{L}{\sigma}$ is large, the performances of these variance reduction methods may degenerate considerably. In view of this, there have been many schemes that incorporate second-order information into the variance reduction schemes. In \cite{gonen2016solving}, the problem data is first transformed by linear sketching in order to decrease the condition number, then SVRG is applied. However, the strategy is only proposed for ridge regression and it is unclear whether it can be applied to other problems.

A larger family of algorithms, called Stochastic Quasi-Newton (SQN) methods, apply to more general settings. The idea is to first sample one or a few Hessian-vector products, then perform a L-BFGS type update on the approximate Hessian inverse $H_k$ \cite{byrd2016stochastic, moritz2016linearly, gower2016stochastic}, then $H_k$ is applied to the SVRG-type stochastic gradient as a preconditioner. That is,
\[
w_{t+1}=w_t-\eta H_k\tilde{\nabla}_t,
\]
where $\tilde{\nabla}_t$ is a variance-reduced stochastic gradient.

Linear convergence is established and competitive numerical performances are observed for SQN methods. However, the theoretical linear rate depends on the condition number of the approximate Hessian, which again depends poorly on the condition number of the objective, so it is not clear whether they are faster than SVRG in general. Furthermore, they do not support nondifferentiable regularizers nonconvexity of individual $f_i$. Recently, the first issue is partially resolved in \cite{lin2018inexact}, where the algorithm is at least as fast as SVRG. To deal with the second issue, \cite{wang2018stochastic} applied a $H_k-$preconditioned proximal mapping of $\psi$ after $H_k$ is applied to the variance reduced stochastic gradient, but in order to evaluate this mapping efficiently, $H_k$ is required to be of the symmetric rank-one update form $\tau I_d+u u^T$, where $I_d\in\R^{d\times d}$ is the identity matrix and $u\in\Rd$. However, $H_k$ is still ill-conditioned with a conditioner number of order $\cO(\frac{1}{\varepsilon})$, therefore only a gradient complexity of order $\cO\big((n+\kappa\frac{1}{\varepsilon})\ln(\frac{1}{\varepsilon})\big)$ can be guaranteed.

Another way of exploiting second-order information is to cyclically calculate one individual Hessian $\nabla^2f_i$ (or an approximation of it) \cite{rodomanov2016superlinearly,mokhtari2018iqn}, linear and locally superlinear convergence are established. However, they require at least an $O(n)$ amount of memory to store the local variables, which will be substantial when $n$ is large.

Aside from exploiting second-order information, it is also possible to apply Nesterov-type acceleration to SVRG. Recently, Katyusha \cite{allen2017katyusha} and Katyusha X \cite{allen2018katyusha} are developed in this spirit. Katyusha X also applies to the sum-of-nonconvex setting where each $f_i$ can be nonconvex. There are also ``Catalyst'' accelerated methods \cite{lin2015universal}, where a small amount of strong convexity $\frac{c}{2}\|x-y^k\|^2$ is added to the objective and is minimized inexactly at each step, then Nesterov acceleration is applied.  However, Catalyst methods have an additional $\ln k$ factor in gradient complexity over Katyusha and Katyusha X.

\subsection{Our Contributions}
\begin{enumerate}
    \item We propose to accelerate SVRG and Katyusha X by a \textit{fixed} preconditioner, as opposed to time-varying preconditioners in SQN methods. And the subproblems are solved with \textit{fixed} number of simple subroutines.
    \item If the preconditioner captures the second order information of $f$, then there will be significant accelerations. With a good preconditioner $M$, when $\kappa_f\in(n^{\frac{1}{2}}, n^{2}d^{-2})$, Algorithm \ref{alg: inexact preconditioned svrg} and Algorithm \ref{alg: inexact preconditioned Katyusha X} are $\cO(\frac{n^{\frac{1}{2}}}{\kappa_f})$ and $\cO(\sqrt{\frac{n^{\frac{1}{2}}}{\kappa_f}})$ times faster than SVRG and Katyusha X in terms of gradient complexity, respectively. When $\kappa_f>n^2d^{-2}$, these numbers become $\cO(\frac{d}{\sqrt{n\kappa_f}})$ and $\cO(\frac{d}{n^{\frac{3}{4}}})$. We also demonstrate these accelerations for Lasso and Logistic regression.
    \item Our acceleration applies to the sum-of-nonconvex setting, where $f(x)$ in \eqref{equ: objective} is strongly convex, but each individual $f_i$ can be nonconvex. We also allow a nondifferentiable regularizer $\psi(x)$.
\end{enumerate}

\section{Preliminaries and Assumptions}
\label{sec: preliminaries}
Throughout this paper, we use $\|\cdot\|$ for  $\ell_2-$norm and $\langle\cdot, \cdot\rangle$ for  dot product, $\|\cdot\|_1$ denotes the $\ell_1-$norm. 

The preconditioner $M\succ 0$ is a  symmetric, positive definite matrix. We write $\lambda_{\text{min}}(M)$ and $\lambda_{\text{max}}(M)$ as the smallest and the largest eigenvalues of $M$, respectively, and  $\kappa(M)\coloneqq\frac{\lambda_{\text{max}}(M)}{\lambda_{\text{min}}(M)}$ as the condition number of $M$. For $M\succ 0$, let $\|\cdot\|_M$ and $\langle \cdot, \cdot \rangle_M$ denote the norm and inner product induced by $M$, respectively, i.e., $\|x\|_M=\sqrt{x^T M x}, \langle x,y\rangle_M=x^T M y$.

We use $\lceil\cdot\rceil$ to denote the ceiling function. For $r\in(0,1]$, $N \sim$ \,\,\textbf{Geom}$(r)$ denotes a random variable $N$ that obeys the geometric distribution, i.e., $N=k$ with probability $(1-r)^k r$ for $k\in\mathbb{N}$. We have $\E[N]=\frac{1-p}{p}$. 
\begin{defn}
\label{def: smoothness}
We say that $f: \Rd \rightarrow \R$ is $L_f-$smooth, if it is differentiable and satisfies
\[
f(y)\leq f(x)+\langle \nabla f(x), y-x\rangle+\frac{L_f}{2}\|y-x\|^2, \forall x, y \in\Rd.
\]
We say that $f: \Rd \rightarrow \R$ is $L^M_f-$smooth under $\|\cdot\|_M$, if it is differentiable and satisfies
\[
f(y)\leq f(x)+\langle \nabla f(x), y-x\rangle+\frac{L_f^M}{2}\|y-x\|^2_M, \forall x, y \in\Rd.
\]
\end{defn}
\begin{defn}
\label{def: strong convexity}
We say that $f$ is $\sigma_f-$strongly convex, if 
\[
f(y)\geq f(x)+\langle \nabla f(x), y-x\rangle+\frac{\sigma_f}{2}\|y-x\|^2, \forall x, y \in\Rd.
\]
We say that $f$ is $\sigma^M_f-$strongly convex under $\|\cdot\|_M$, if 
\[
f(y)\geq f(x)+\langle \nabla f(x), y-x\rangle+\frac{\sigma_f^M}{2}\|y-x\|^2_M, \forall x, y \in\Rd.
\]
\end{defn}

$L^M_f-$smoothness under $\|\cdot\|_M$ is equivalent to $\|\nabla f_i(x)-\nabla f_i(y)\|_{M^{-1}}\leq L^M_f\|x-y\|_{M}$. Also, $\sigma^M_f-$strong convexity is equivalent to $\|\nabla f_i(x)-\nabla f_i(y)\|_{M^{-1}}\geq\sigma^M_f\|x-y\|_M$. Cf. Section 2 of \cite{Shalev-ShwartzZhang2016_accelerated}.

\begin{defn}
We define the condition number of $f$ under $\|\cdot\|_M$ as $\kappa^M_f\coloneqq\frac{L^M_f}{\sigma^M_f}$.
\end{defn}

When $M=I$, 
we have $\kappa^M_f=\kappa_f\coloneqq\frac{L_f}{\sigma_f}$. 

In this paper, we will choose $M$ such that $\kappa^M_f\ll \kappa$. For example, if $f(x)=\frac{1}{2}x^T Q x$ where $Q\succ 0$ is ill-conditioned, by choosing $M=Q$ we have
\[
\|\nabla f(x)-\nabla f(y)\|_{M^{-1}}\equiv \|x-y\|_Q,
\]
which tells us that $L^M_f=\sigma^M_f=1$ and $\kappa^M_f=1$, while $\kappa_f=\kappa(Q)\gg 1$. That is, under $Q-$metric, $f(x)$ has a much smaller condition number and can be minimized easily.

\begin{defn}
\label{def: subdifferential}
For a proper closed convex function $\phi:\Rd\rightarrow\R\cup\{+\infty\}$, its subdifferential at $x\in\dom f$ is written as
$$
\partial \phi(x)=\{v\in \Rd\,|\, \phi(z)\geq \phi(x)+\langle v, z-x\rangle\,\,\forall z\in\Rd\}.$$
\end{defn}
\begin{defn}
\label{def: proximal operator}
For a proper closed convex function $\phi:\Rd\rightarrow\RR$, its $M-$preconditioned proximal mapping with step size $\eta>0$ is defined by
\[
\prox^M_{\eta \psi}(x)=\argmin_{y\in\Rd}\{\psi(y)+\frac{1}{2\eta}\|x-y\|_M^2\}.
\]
\end{defn}
When $M=I$, this reduces to the classical proximal mapping.

Finally, let us list the assumptions that will be effective throughout this paper.
\begin{asmp}
\label{Assumption 1}
In the objective function \eqref{equ: objective}, 
\begin{enumerate}
    \item Each $f_i(x)$ is $L_f-$smooth and $L^M_f-$smooth under \\$\|\cdot\|_M$.
    \item $f(x)$ is $\sigma_f-$strongly convex, and $\sigma^M_f-$strongly convex under $\|\cdot\|_M$, where $\sigma_f>0$ and $\sigma^M_f>0$.
    \item The regularization term $\psi(x)$ is proper closed convex and $\prox_{\eta\psi}$ is easy to compute.
\end{enumerate}
\end{asmp}

\begin{rem}
\begin{enumerate}
    \item In Assumption \ref{Assumption 1}, we only require $f(x)=\frac{1}{n}\sum_{i=1}^n f_i(x)$ to be strongly convex, while each $f_i(x)$ can be nonconvex.
    \item Several common choices of regularizers have simple proximal mappings. For example, when $\psi(x)=\lambda \|\cdot\|_1$ with $\lambda>0$, $\prox_{\eta \psi}$ can be computed component wise as
    \[
    \prox_{\eta \psi}(x)=\text{sign}(x)\max\{|x|-\eta\lambda, 0\}.
    \]
\end{enumerate}
\end{rem}

\section{Proposed Algorithms}
\label{sec: proposed algorithms}

As discussed in Sec. \ref{sec: introduction}, SVRG and Katyusha X suffer from ill-conditioning like other first order methods. In this section, we propose to accelerate them by applying inexact preconditioning. Let us illustrate the idea as follows,

\begin{enumerate}
    \item We would like to apply a preconditioner $M\succ 0$ to the gradient descent step in SVRG. i.e.,
    \begin{align}
    \label{equ: exact update}
    w_{t+1}&=\prox^M_{\eta\psi}(w_t-\eta M^{-1} \tilde{\nabla}_t)\nonumber\\
    &=\argmin_{y\in\mathbb{R}^d}\{\psi(y)+\frac{1}{2\eta}\|y-w_t\|_M^2+\langle \tilde{\nabla}_t, y\rangle\}.
    \end{align}
    where $\tilde{\nabla}_t$ is a variance-reduced stochastic gradient. When $\psi=0$ and this minimization is solved exactly, we have $w_{t+1}=w_t-\eta M^{-1}\tilde{\nabla}_t$, which is a preconditioned gradient update.
    \item However, solving \eqref{equ: exact update} exactly may be expensive and impractical. In fact it suffices to solve it \textit{highly inexactly} by \textit{fixed} number of simple subroutines.
\end{enumerate}
We summarize the resulted algorithm in Algorithm \ref{alg: inexact preconditioned svrg} and call it Inexact Preconditioned(IP-) SVRG. Compared to SVRG, the only difference lies in line $7$. 

\begin{algorithm}[H]    
\caption{$\text{Inexact Preconditioned SVRG(iPreSVRG)}$}
\label{alg: inexact preconditioned svrg}
    \textbf{Input:} $F(\cdot)=\psi(\cdot)+\frac{1}{n}\sum_{i=1}^{n} f_i(\cdot)$, initial vector $x^0$, step size $\eta>0$, preconditioner $M\succ 0$, number of epochs $K$.\\
    \textbf{Output:} vector $x^K$
    \begin{algorithmic}[1]
        \FOR{$k\leftarrow 0,...,K-1$}{}
        \STATE{$D^k \sim  \text{\textbf{Geom}}(\frac{1}{m})$;}
        \STATE{$w_0\leftarrow x^k$, $g\leftarrow \nabla f(x^k)$;}
        \FOR{$t \leftarrow 0,...,D^k$}{}
        \STATE{pick $i_t\in\{1,2,...,n\}$ uniformly at random;}
        \STATE{$\tilde{\nabla}_t=g+\big(\nabla f_{i_t}(w_t)-\nabla f_{i_t}(w_0)\big);$}
        \STATE{$w_{t+1}\approx \argmin_{y\in\mathbb{R}^d}\{\psi(y)+\frac{1}{2\eta}\|y-w_t\|_M^2+\langle \tilde{\nabla}_t, y\rangle\};$}
        \ENDFOR
        \STATE{$x^{k+1}\leftarrow w_{D+1};$}
        \ENDFOR
    \end{algorithmic}
\end{algorithm}
\begin{rem}
\begin{enumerate}
\item In line $2$, the epoch length $D^k$ obeys a geometric distribution and $\mathbb{E}[m^k]=m-1$, this is for the purpose of simplifying analysis (motivated by \cite{lei2017less,allen2018katyusha}), in practice one can just set $D^k=m-1$. In our experiments, this still brings significant accelerations.
\item The choice of $m$ affects the performance. Intuitively, a larger $m$ means more gradient evaluations per epoch, but also more progress per epoch. Theoretically, we show that $m=\lceil\frac{n}{1+pd}\rceil$ gives faster convergence than SVRG, where $p$ is the number of subroutines used in Line $7$.
\item In line $6$, one can also sample a batch of gradients instead of one. It is straightforward to generalize our convergence results in Sec. \ref{sec: main theory} to this setting.
\item If $M=I$, line $7$ reduces to
\[
w_{t+1}=\prox_{\eta \psi}(w_t-\eta \tilde{\nabla}_t),
\]
and Algorithm \ref{alg: inexact preconditioned svrg} reduces to SVRG.
\end{enumerate}
\end{rem}
For $M\not\propto I$, line $7$ contains an optimization problem that may not have a closed form solution:
\begin{equation}
\argmin_{y\in\mathbb{R}^d}\{\psi(y)+\frac{1}{2\eta}\|y-w_t\|_M^2+\langle \tilde{\nabla}_t, y\rangle\}.
\label{equ: subproblem}
\end{equation}
To solve it inexactly, we propose to apply \textit{fixed} number of iterations of some simple subroutines, which are initialized at $w_t$. This procedure is summarized in Procedure \ref{alg: fixed number of inner loops}.

\begin{procedure}[H]
\caption{$\text{Procedure for solving \eqref{equ: subproblem} inexactly}$}
\label{alg: fixed number of inner loops}
    \textbf{Input:} Iterator $S$, iterator step size $\gamma>0$, number of iterations $p\geq 1$, problem data $\eta>0, w_t, M\succ 0, \tilde{\nabla}_t, \psi(\cdot)$.\\
    \textbf{Output:} vector $w_{t+1}$
    \begin{algorithmic}[1]
    \STATE{$w^{0}_{t+1}\leftarrow w_t;$}
        \FOR{$i\leftarrow 0,...,p-1$}{}
        \STATE{$w^{i+1}_{t+1}=S(w^i_{t+1}, \eta, M, \tilde{\nabla}_t, \psi)$;}
        \ENDFOR
    \STATE{$w_{t+1}\leftarrow w^p_{t+1}$;}
    \end{algorithmic}
\end{procedure}

\begin{rem}
In Procedure \ref{alg: fixed number of inner loops}, there are many choices for the iterator $S$, for example, one can use proximal gradient, FISTA \cite{beck2009fast} (or equivalently, Nesterov acceleration \cite{nesterov2013introductory}), and FISTA with restart \cite{o2015adaptive}. Under these choices, line $3$ is easy to compute. For example, when $S$ is the proximal gradient step, line $3$ of Procedure \ref{alg: fixed number of inner loops} becomes
\[
w^{i+1}_{t+1} = \prox_{\gamma \psi}(w^{i}_{t+1}- \frac{\gamma}{\eta}M(w^{i}_{t+1}-w_t)-\gamma\tilde{\nabla}_t).
\]
\end{rem}
Now, let us also apply the inexact preconditioning idea to Katyusha X (Algorithm 2 of \cite{allen2018katyusha}). Similar to Katyusha X, we first apply a momentum step, then one epoch of iPreSVRG (i.e., line $2\sim 9$ of Algorithm \ref{alg: inexact preconditioned svrg}).
\begin{algorithm}[H]    
\caption{$\text{Inexact Preconditioned Katyusha X(iPreKatX)}$}
\label{alg: inexact preconditioned Katyusha X}
    \textbf{Input:} $F(x)=\psi(x)+\frac{1}{n}\sum_{i=1}^{n} f_i(x)$, initial vector $x^0$, step size $\eta>0$, preconditioner $M\succ 0$, momentum weight $\tau\in(0,1]$, number of epochs $K$.\\
    \textbf{Output:} vector $y^K$
    \begin{algorithmic}[1]
    \STATE{$y_{-1}=y_0\leftarrow x_0;$}
        \FOR{$k\leftarrow 0,...,K-1$}{}
        \STATE{$x_{k+1}\leftarrow \frac{\frac{3}{2}y_k+\frac{1}{2}x_k-(1-\tau)y_{k-1}}{1+\tau}$;}
        \STATE{$y_{k+1}\leftarrow \text{Algorithm \ref{alg: inexact preconditioned svrg}}^{\text{1ep}}(F, M, x_{k+1}, \eta)$;}
        \ENDFOR
    \end{algorithmic}
\end{algorithm}

\begin{rem}
\begin{enumerate}
    \item When $\tau=\frac{1}{2}$, one can show that $x_{k+1}\equiv y_k$, and Algorithm \ref{alg: inexact preconditioned Katyusha X} reduces to Algorithm \ref{alg: inexact preconditioned svrg}.
    \item When $M=I$ and the proximal mapping is solved exactly, Algorithm \ref{alg: inexact preconditioned Katyusha X} reduces to Katyusha X. 
    \item The convergence of Algorithm \ref{alg: inexact preconditioned Katyusha X} is established when $\tau=\frac{1}{2}\sqrt{\frac{1}{2}m\eta\sigma^M_f}$. In practice, we found that many other choices of $\tau$ also work.
\end{enumerate}
\end{rem}

\section{Main Theory}
\label{sec: main theory}
In this section, we proceed to establish the convergence of Algorithm \ref{alg: inexact preconditioned svrg} and Algorithm \ref{alg: inexact preconditioned Katyusha X}. The key idea is that when the preconditioned proximal gradient update in \eqref{equ: subproblem} is solved inexactly as in Procedure \ref{alg: fixed number of inner loops},
the error can be bounded by $\|w_{t+1}-w_{t}\|_M$, under which we can still establish the overall convergence of Algorithm \ref{alg: inexact preconditioned svrg} and Algorithm \ref{alg: inexact preconditioned Katyusha X}. 
Combine this with the fixed number of simple subroutines in Procedure \ref{alg: fixed number of inner loops}, we obtain a much lower gradient complexity when $\kappa_f>n^{\frac{1}{2}}$.

All the proofs in this section are deferred to the supplementary material.

First, Let us analyze the error in the optimality condition of \eqref{equ: subproblem} when it is solved inexactly by FISTA with restart as in Procedure \ref{alg: fixed number of inner loops}. Specifically,



Let $h_1(y)=\psi(y)$ and 
$    h_2(y)=\frac{1}{2\eta}\|y-w_t\|^2_{M}+\langle\tilde{\nabla}, y\rangle,$ then the subproblem \eqref{equ: subproblem} can be written as
\[
\min_{y} \Psi(y)=h_1(y)+h_2(y).
\]

Therefore, FISTA with restart applied to \eqref{equ: subproblem} can be summarized in the following algorithm.
\begin{algorithm}[H]
\caption{FISTA with restart for solving \eqref{equ: subproblem}}
\label{alg: fixed number of FISTA with restart}
\textbf{Input:} Iterator $S$, iterator step size $\gamma>0$, number of iterations $p\geq 1$, problem data $\eta>0, w_t, h_1(y)=\psi(y)$ and 
$    h_2(y)=\frac{1}{2\eta}\|y-w_t\|^2_{M}+\langle\tilde{\nabla}, y\rangle.$
    \begin{algorithmic}[1]
    \STATE{$w^{(0,0)}_{t+1}=u^{(0,1)}_{t+1}\leftarrow w_t, \theta_0 = 1$}
        \FOR{$i\leftarrow 0,...,r-1$}{}
        \FOR{$j\leftarrow 0,..., p_0-1$}{}
        \STATE{$\theta_0=1;$}
        \STATE{$w^{(i,j+1)}_{t+1}=\prox_{\gamma h_1}\big(u^{(i, j+1)}_{t+1}-\gamma \nabla h_2(u^{(i, j+1)}_{t+1})\big)$;}
        \STATE{$\theta_{j+1}=\frac{1+\sqrt{1+4\theta_{j}^2}}{2}$;}
        \STATE{$u^{(i, j+2)}_{t+1}=w^{(i, j+1)}_{t+1}+\frac{\theta_j-1}{\theta_{j+1}}(w^{(i, j+1)}_{t+1}-w^{(i, j)}_{t+1});$}
        \ENDFOR
        \STATE{$w^{(i+1, 0)}_{t+1}=u^{(i+1, 1)}_{t+1}\leftarrow w^{(i, p_0)}_{t+1}$}
        \ENDFOR
    \STATE{$w_{t+1}\leftarrow w^{(r-1,p_0)}_{t+1}$;}
    \end{algorithmic}
\end{algorithm}

\begin{lem}
\label{lem: finite loops of FISTA with restart}
Take Assumption \ref{Assumption 1}. Suppose in Procedure \ref{alg: fixed number of inner loops}, we choose $S$ as the iterator of FISTA with restart\footnotemark[1] every $p_0=\lceil 2e\sqrt{\kappa(M)}\rceil$ steps, with step size $\gamma=\frac{\eta}{\lambda_{\mathrm{max}}(M)}$ and restart it $(r-1)$ times (that is, $p=r p_0$ iterations in total). Then, $w_{t+1}=w_{t+1}^{(r-1, p_0)}$ is an approximate solution to \eqref{equ: subproblem} that satisfies 
\begin{align}
\vz \in &\partial \psi(w_{t+1})+\frac{1}{\eta}M(w_{t+1}-w_t)+\tilde{\nabla}_t+M\varepsilon^p_{t+1},\label{equ: inexact optimality condition}\\
&\|\varepsilon^p_{t+1}\|_M\leq\frac{c(p)}{\eta}\|w_{t+1}-w_t\|_M,\label{equ: error bound}
\end{align}
where
\begin{align*}
    c(p)=14\kappa(M)\frac{\tau^p}{1-\tau^p},
\end{align*}
and 
\begin{align*}
\tau&=(\frac{4\kappa(M)}{p_0^2})^{\frac{1}{2 p_0}}\leq\exp(-\frac{1}{2e\sqrt{\kappa(M)}+1})<1.
\end{align*}
\end{lem}
\footnotetext[1]{FISTA with restart can be replaced with any iterator with Q-linear convergence on the iterates. In our experiments, FISTA also works, and a simple choice of $p=20$ is enough.}


With Lemma \ref{lem: finite loops of FISTA with restart}, the overall convergences of Algorithm \ref{alg: inexact preconditioned svrg} and \ref{alg: inexact preconditioned Katyusha X} can be established. The analysis is similar to that of \cite{allen2018katyusha}.

\begin{thm}
\label{thm: convergence of inexact preconditioned svrg}
Under Assumption \ref{Assumption 1}, let $x^*=\argmin_x F(x)$, $64\kappa^M_f c^2(p)\leq 1$, $\eta\leq \frac{1}{2\sqrt{m}L^M_f}$, and $m\geq 4$. Then the iPreSVRG in Algorithm \ref{alg: inexact preconditioned svrg} satisfies
\begin{align}
    \label{equ: inexact preconditioned svrg linear convergence}
    \mathbb{E}[F({x}^{k})-F(x^*)]\leq \cO\big((\frac{1}{1+\frac{1}{4}m\eta\sigma^M})^k\big).
\end{align}
\end{thm}


\begin{thm}
\label{thm: convergence of inexact preconditioned Katyusha X}
Under Assumption \ref{Assumption 1}, let $x^*=\argmin_x F(x)$, $64\kappa^M_f c^2(p)\leq 1$, $\tau=\frac{1}{2}\sqrt{\frac{1}{2}m\eta \sigma^M_f}$, $\eta\leq \frac{1}{2\sqrt{m}L^M_f}$, and $m\geq 4$. Then the iPreKatX in Algorithm \ref{alg: inexact preconditioned Katyusha X} satisfies
\begin{align}
    \label{equ: inexact preconditioned Katyusha X linear convergence}
    \mathbb{E}[F({x}^{k})-F(x^*)]\leq \cO\big((\frac{1}{1+\frac{1}{2}\sqrt{\frac{1}{2}m\eta\sigma^M}})^k\big).
\end{align}
\end{thm}
\begin{rem}
\label{rem: M=I gives previous results}
When $M=I$, we have $c(p)=0$, and Theorems \ref{thm: convergence of inexact preconditioned svrg} and \ref{thm: convergence of inexact preconditioned Katyusha X} recovers the Theorems D.1 and 4.3 of \cite{allen2018katyusha}. 
\end{rem}

In Theorems \ref{thm: convergence of inexact preconditioned svrg} and \ref{thm: convergence of inexact preconditioned Katyusha X}, we need the number of simple subroutines $p$ to be large enough such that $64\kappa^M_f c^2(p)\leq 1$, the following Lemma provides a sufficient condition for this.

\begin{lem}
\label{lem: choice of p with FISTA with restart}
If the subproblem iterator $S$ in Procedure \ref{alg: fixed number of inner loops} is FISTA with restart every $p_0=\lceil 2e\sqrt{\kappa(M)}\rceil$ steps, and with step size $\gamma=\frac{\eta}{\lambda_{\mathrm{max}}(M)}$, then, in order for  $64\kappa^M_f c^2(p)\leq 1$ to hold, it suffices to choose
\begin{align}
p&= (2e\sqrt{\kappa(M)}+1)\ln\frac{\sqrt{\kappa^M_f}\kappa(M)+\sqrt{c_1}}{c_1}\label{equ: choice of p FISTA with restart}\\
&=\cO\bigg(\sqrt{\kappa(M)}\ln\big(\sqrt{\kappa^M_f}\kappa(M)\big)\bigg)\nonumber
\end{align}
where $c_1=\frac{1}{64*14^2}$.
\end{lem}

With \eqref{equ: inexact preconditioned svrg linear convergence}, \eqref{equ: inexact preconditioned Katyusha X linear convergence}, and \eqref{equ: choice of p FISTA with restart}, we can now calculate the gradient complexities of Algorithm \ref{alg: inexact preconditioned svrg} and Algorithm \ref{alg: inexact preconditioned Katyusha X}, but let us first do that for SVRG and Katyusha X.

In Assumption \ref{Assumption 1}, we have assumed that $\prox_{\eta \psi}(\cdot)$ is cheap to evaluate, therefore, each epoch of SVRG needs $n+m$ gradient evaluations, which is also true for Katyusha X. As a result, the gradient complexity for SVRG and Katyusha X to reach $\varepsilon-$suboptimality are:
\begin{align}
C_1(m, \varepsilon)=\cO(\frac{n+m}{\ln(1+\frac{1}{4}m\eta \sigma)}\ln{\frac{1}{\varepsilon}}),\label{equ: svrg gradient complexity}\\
C_2(m, \varepsilon)=\cO(\frac{n+m}{\ln(1+\frac{1}{2}\sqrt{\frac{1}{2}m\eta \sigma})}\ln{\frac{1}{\varepsilon}}).\label{equ: Katyusha X gradient complexity}
\end{align}
For Algorithm \ref{alg: inexact preconditioned svrg} and Algorithm \ref{alg: inexact preconditioned Katyusha X}, each iteration in Procedure \ref{alg: fixed number of inner loops} is at most as expensive as $d$ gradient computations\footnotemark[1] and is operated $p$ times, therefore, one epoch of iPreSVRG/iPreKatX needs at most $n+(1+pd)m$ gradient computations.

Consequently, we can write the the gradient complexity for Algorithm \ref{alg: inexact preconditioned svrg} and Algorithm \ref{alg: inexact preconditioned Katyusha X} to reach $\varepsilon-$suboptimality as:
\begin{align}
C'_1(m, \varepsilon)=\cO(\frac{n+(1+pd)m}{\ln(1+\frac{1}{4}m\eta \sigma^M)}\ln{\frac{1}{\varepsilon}}),\label{equ: inexact preconditioned svrg gradient complexity}\\
C'_2(m, \varepsilon)=\cO(\frac{n+(1+pd)m}{\ln(1+\frac{1}{2}\sqrt{\frac{1}{2}m\eta \sigma^M})}\ln{\frac{1}{\varepsilon}}).\label{equ: inexact preconditioned Katyusha X gradient complexity}
\end{align}
\begin{rem}
\begin{enumerate}
    \item According to Lemma \ref{lem: choice of p with FISTA with restart}, when $S$ is FISTA with restart, it suffices to choose $p$ by \eqref{equ: choice of p FISTA with restart}.
    \item When the preconditioner $M$ is chosen appropriately, the step size $\eta$ in \eqref{equ: inexact preconditioned svrg gradient complexity} and \eqref{equ: inexact preconditioned Katyusha X gradient complexity} can be much larger than that of \eqref{equ: svrg gradient complexity} and \eqref{equ: Katyusha X gradient complexity}.
\end{enumerate}
\end{rem}

Finally, we can compare $C_1(m, \varepsilon)$, $C_2(m, \varepsilon)$ with $C'_1(m, \varepsilon)$, $C'_2(m, \varepsilon)$, respectively. It turns out that there is a significant speedup when $\kappa>n^{\frac{1}{2}}$.

\footnotetext[1]{For each iteration of Procedure \ref{alg: fixed number of inner loops}, the most expensive step is multiplying $M$ to some vector, which is often cheaper than $d$ gradient computations.}

\begin{thm}
\label{thm: inexact preconditioned svrg vs svrg}
Take Assumption \ref{Assumption 1}. Let the iterator $S$ in Procedure \ref{alg: fixed number of inner loops} be FISTA with restart, and an appropriate preconditioner $M$ is chosen such that $\kappa_f$ and $\kappa(M)$ are of the same order, and $\kappa^M_f$ is small compared to them, then
\begin{enumerate}
\item if $\kappa_f>n^{\frac{1}{2}}$ and $\kappa_f<n^{2}d^{-2}$, then
    \begin{align}
    \label{equ: compare svrg gradient complexities case 1}
    \frac{\min_{m\geq 1}C_1'(m,\varepsilon)}{\min_{m\geq 1}C_1(m,\varepsilon)}\leq
    \cO\big(\frac{n^{\frac{1}{2}}}{\kappa_f}\big).
    \end{align}
    \item if $\kappa_f > n^{\frac{1}{2}}$ and $\kappa_f>n^{2}d^{-2}$, then
    \begin{align}
    \label{equ: compare svrg gradient complexities case 2}
    \frac{\min_{m\geq 1}C_1'(m,\varepsilon)}{\min_{m\geq 1}C_1(m,\varepsilon)}\leq \cO(\frac{d}{\sqrt{n\kappa_f}}).
    \end{align}
\end{enumerate}
\end{thm}
\begin{thm}
\label{thm: inexact preconditioned Katyusha X vs Katyusha X}
Take Assumption \ref{Assumption 1}. Let the iterator $S$ in Procedure \ref{alg: fixed number of inner loops} be FISTA with restart, and an appropriate preconditioner $M$ is chosen such that $\kappa_f$ and $\kappa(M)$ are of the same order, and $\kappa^M_f$ is small compared to them, then
\begin{enumerate}
\item if $\kappa_f>n^{\frac{1}{2}}$ and $\kappa_f<n^{2}d^{-2}$, then
    \begin{align}
    \label{equ: compare katyusha X complexities case 1}
    \frac{\min_{m\geq 1}C_2'(m,\varepsilon)}{\min_{m\geq 1}C_2(m,\varepsilon)}\leq
    \cO\big(\sqrt{\frac{n^{\frac{1}{2}}}{\kappa_f}}\big).
    \end{align}
    \item If $\kappa_f>n^{\frac{1}{2}}$ and $\kappa_f>n^{2}d^{-2}$, then
    \begin{align}
    \label{equ: compare katyusha X gradient complexities case 2}
    \frac{\min_{m\geq 1}C_2'(m,\varepsilon)}{\min_{m\geq 1}C_2(m,\varepsilon)}\leq
    \cO(\frac{d}{n^{\frac{3}{4}}}).
    \end{align}
\end{enumerate}
\end{thm}

In Section \ref{sec: experiments}, we provide practical choices of $M$ for Lasso and Logistic regression.

\section{Experiments}
\label{sec: experiments}

To investigate the practical performance of Algorithms \ref{alg: inexact preconditioned svrg} and \ref{alg: inexact preconditioned Katyusha X}, we test on three problems: Lasso, logistic regression, and a synthetic sum-of-nonconvex problem. For the first two, each function in the finite sum is convex. To guarantee that the objective is strongly convex, a small $\ell_2-$regularization is added to Lasso and logistic regression. 

In the following, we compare SVRG, iPreSVRG, Katyusha X, and iPreKatX on four datasets from LIBSVM\footnote{\url{https://www.csie.ntu.edu.tw/~cjlin/libsvmtools/datasets/}}: \texttt{w1a.t} (47272 samples, 300 features), \texttt{protein} (17766 samples, 357 features), \texttt{cod-rna.t} (271617 samples, 8 features), \texttt{australian} (690 samples, 14 features), and one synthetic dataset. The implementation settings are listed below,
\begin{enumerate}
    \item We choose the epoch length $m=100$ in all experiments, since we found that the choices $m\in\{\frac{n}{4}, \frac{n}{2}, n\}$ need more gradient evaluations.
    \item For iPrePDHG and iPreKatX, we use FISTA as the subproblem iterator $S$. If the preconditioner $M$ is diagonal, then the number of subroutines for solving the subproblem is $p=1$, if not, then we set $p=20$.
    \item In all the experiments, we tune the step size $\eta$ and momentum weight $\tau$ to their optimal.
    \item All algorithms are initialized at $x^0=\vz$.
    \item All algorithms are implemented in Matlab R2015b. To be fair, except for the subproblem routines for inexact preconditioning, the other parts of the code are identical in all algorithms. The experiments are conducted on a Windows system with Intel Core i7 2.6 GHz CPU.
    The code is available at: 
    \[\text{\url{https://github.com/uclaopt/IPSVRG}}.\]
\end{enumerate}
\subsection{Lasso}
We formulate Lasso as
\begin{equation}
\label{equ: lasso}
    \mathop{\mathrm{minimize}}_{x\in\Rd} ~\frac{1}{2n}\sum_{i=1}^n ~(a_i^Tx-b_i)^2 + \lambda_1\|x\|_1+\lambda_2\|x\|_2^2,
\end{equation}
where $a_i\in\mathbb{R}^d$ are feature vectors and $b_i\in\mathbb{R}$ are labels.
Note that the first term is equivalent to $\frac{1}{2n}\|Ax-b\|^2$, where $A=(a_1,a_2,...,a_n)^T\in\mathbb{R}^{n\times d}$ and $b=(b_1, b_2, \dots, b_n)\in\mathbb{R}^n$.

For Lasso as in \eqref{equ: lasso}, we provide two choices of preconditioner $M$,
\begin{enumerate}
    \item When $d$ is small, we choose
    \begin{align*}
    M_1=\frac{1}{n}A^TA,
\end{align*}
this is the exact Hessian of the smooth part of the objective.
    \item When $d$ is large and $A^TA$ is diagonally dominant, we choose
    \begin{align*}
        M_2=\frac{1}{n}\text{diag}(A^T A)+\alpha I,
    \end{align*}
    where $\alpha>0$. In this case, the subproblem \eqref{equ: subproblem} can be solved exactly with $p=1$ iteration.
\end{enumerate}
Our numerical results are presented in the following figures. We didn't observe significant accelerations of Katyusha X over SVRG and iPreKatX over iPrePDHG, and we suspect the reason is that $m=100$ and the optimal choices of step size $\eta$ make $m\eta\sigma_f>1$ or $m\eta\sigma^M_f>1$, thus the complexity in \eqref{equ: Katyusha X gradient complexity} and \eqref{equ: inexact preconditioned Katyusha X gradient complexity} are not better than \eqref{equ: svrg gradient complexity} and \eqref{equ: inexact preconditioned svrg gradient complexity}, respectively.
\begin{figure}[H]
    \centering
    \includegraphics[width=.5\textwidth]{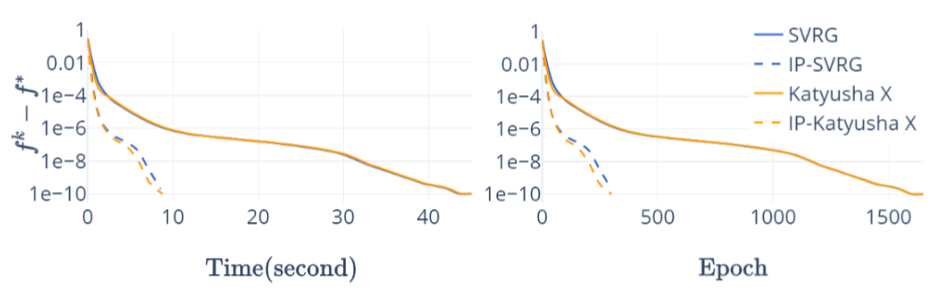}
    \caption{Lasso on \texttt{w1a.t}, $(n,d)=(47272, 300)$, $\lambda_1=10^{-3}, \lambda_2 = 10^{-8}$. For iPreSVRG and iPreKatX: $\eta_1=0.005$; For SVRG and Katyusha X: $\eta_2=0.08$; For Katyusha X and iPreKatX: $\tau=0.45$, $M=M_2$ with $\alpha = 0.01$.}
    \label{fig: new lasso w1at}
\end{figure}
\begin{figure}[H]
    \centering
    \includegraphics[width=.5\textwidth]{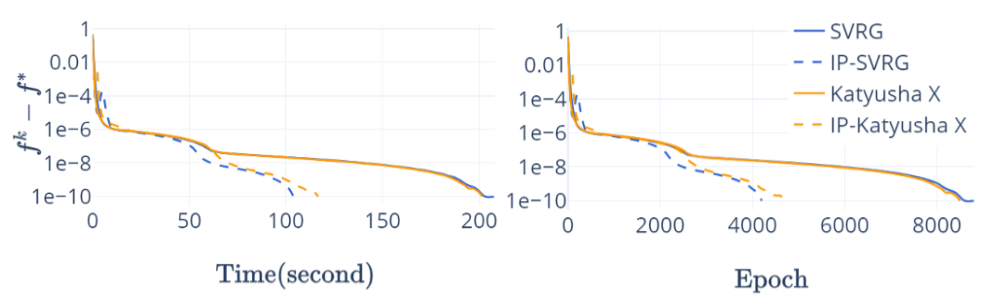}
    \caption{Lasso on \texttt{protein}, $(n,d)=(17766, 357)$, $\lambda_1=10^{-4}, \lambda_2 = 10^{-6}$,  $\eta_1=0.008$, $\eta_2=0.2$, $\tau=0.2$, $M=M_2$ with $\alpha =0.008$.}
    \label{fig: new lasso protein}
\end{figure}
\begin{figure}[H]
    \centering
    \includegraphics[width=.5\textwidth]{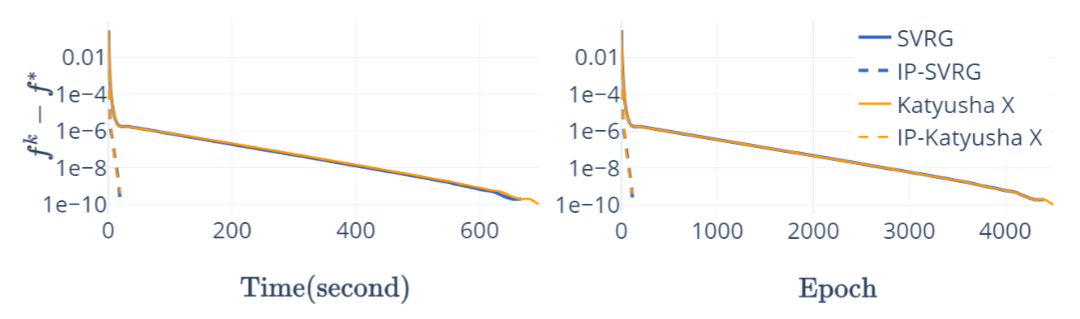}
    \caption{Lasso on \texttt{cod-rna.t}, $(n,d)=(271617, 8)$, $\lambda_1=10^{-2}, \lambda_2 = 1$, $\eta_1=1$, $\eta_2=5\times 10^{-6}$, $\tau=0.45$, $M=M_1$, subproblem iterator step size $\gamma=3\times 10^{-6}$.}
    \label{fig: new lasso codrnat}
\end{figure}

\begin{figure}[H]
     \centering
     \includegraphics[width=.5\textwidth]{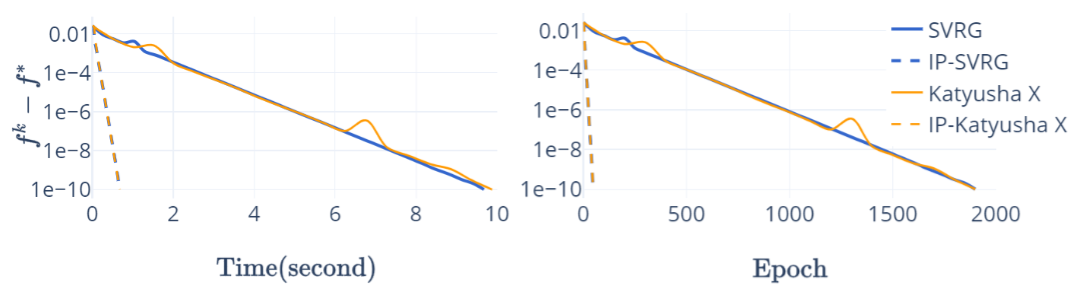}
     \caption{ Lasso on \texttt{australian}, $(n,d)=(690,14)$, $\lambda_1=2, \lambda_2 = 10^{-8}$, $\eta_1=0.01$, $\eta_2=8\times10^{-10}$, $\tau=0.49$, $M=M_1$,$\gamma=5\times 10^{-10}$.}
     \label{fig: new lasso australian}
\end{figure}

\subsection{Logistic Regression}
We formulate Logistic regression as 
\begin{equation}
\label{equ: logistic}
    \mathop{\mathrm{minimize}}_{x\in\Rd} ~\frac{1}{n}\sum_{i=1}^n \ln\big(1+\exp(-b_i \cdot a_i^T x)\big) + \lambda_1\|x\|_1+\lambda_2\|x\|_2^2,
\end{equation}
where again $a_i\in\mathbb{R}^d$ are feature vectors and $b_i\in\mathbb{R}$ are labels.

For Logistic regression as in \eqref{equ: logistic}, the Hessian of the smooth part can be expressed as
\[
H=\frac{1}{n}\sum_{i=1}^n\frac{\exp(-b_i a_i^T x)}{\big(1+\exp(-b_i a_i^T x)\big)^2}b_i^2a_i a_i^T\preccurlyeq \frac{1}{4n} B^T B,
\]
where $B=\text{diag}(b)A=\text{diag}(b)(a_1,a_2,...,a_n)^T$.
Inspired by this\footnotemark[1], we provide two choices of preconditioner $M$, 
\footnotetext[1]{Here is a heuristic justification: By Definition \ref{def: smoothness} we know that $L^M_{f}=1$; Since $\frac{\exp(-b_i a_i^T x)}{\big(1+\exp(-b_i a_i^T x)\big)^2}\rightarrow 0$ only when $x$ is unbounded, we know that if the iterates $x^k$ of our algorithms are bounded, then $H(x^k)\succcurlyeq \frac{c}{n} B^TB$ for some $c>0$, which gives $\sigma^M_f=4c$ according to Definition \ref{def: strong convexity}. When $c$ is not too small, one can expect $\kappa^M_f=\frac{1}{4c}\ll \kappa_f$.}
\begin{enumerate}
    \item When $d$ is small, we choose
    \begin{align*}
    M_1=\frac{1}{4n}B^TB.
    \end{align*}
    \item When $d$ is large and $B^TB$ is diagonally dominant, we choose
    \begin{align*}
        M_2=\frac{1}{4n}\text{diag}(B^T B)+\alpha I,
    \end{align*}
    where $\alpha>0$. In this case, the subproblem \eqref{equ: subproblem} can be solved exactly with $p=1$ iteration.
\end{enumerate}
Our results are presented in the following figures, again, we didn't observe a significant acceleration of Katyusha X over SVRG and iPreKatX over iPrePDHG, due to the same reason mentioned in the last subsection.
\begin{figure}[H]
    \centering
    \includegraphics[width=.5\textwidth]{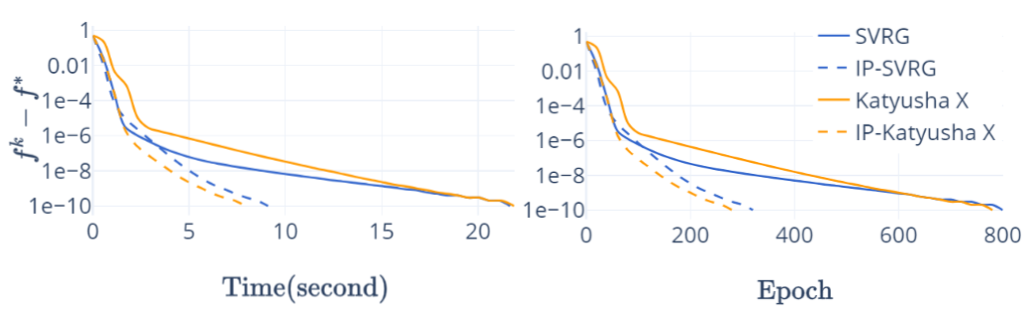}
    \caption{Logistic regression on \texttt{w1a.t}, $(n,d)=(47272, 300)$, $\lambda_1=5\times 10^{-4}, \lambda_2 = 10^{-8}$, $\eta_1=0.06$, $\eta_2=4$, $\tau=0.4$, $M=M_2$ with $\alpha =0.005$.}
    \label{fig: logistic w1at}
\end{figure}

\begin{figure}[H]
    \centering
    \includegraphics[width=.5\textwidth]{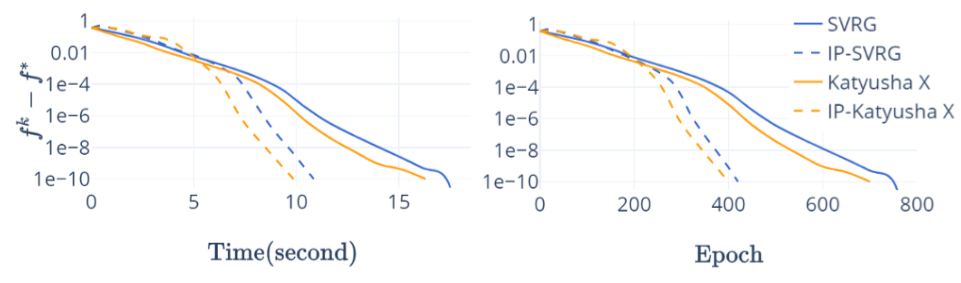}
    \caption{Logistic regression on \texttt{protein}, $(n,d)=(17766, 357)$, $\lambda_1=10^{-4}, \lambda_2 = 10^{-8}$, $\eta_1=1.5$, $\eta_2=10$, $\tau=0.3$, $M=M_2$ with  $\alpha =0.05$.}
    \label{fig: logistic protein}
\end{figure}

\begin{figure}[H]
    \centering
    \includegraphics[width=.5\textwidth]{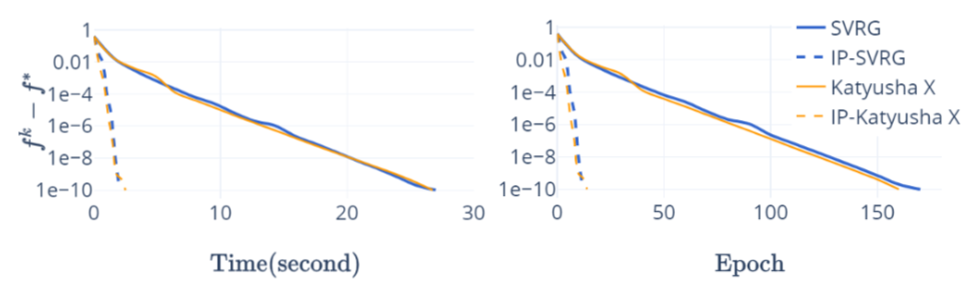}
    \caption{Logistic regression on \texttt{cod-rna.t}, $(n,d)=(271617, 8)$, $\lambda_1=0.1, \lambda_2 = 10^{-8}$, $\eta_1=1$, $\eta_2=3\times 10^{-5}$, $\tau=0.4$, $M=M_1$, $\gamma=2\times 10^{-5}$.}
    \label{fig: logistic codrnat}
\end{figure}

\begin{figure}[H]
    \centering
    \includegraphics[width=.5\textwidth]{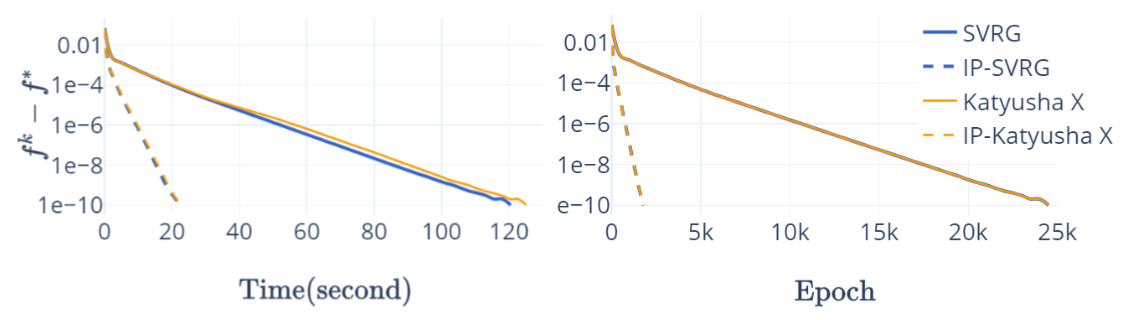}
    \caption{Logistic regression on \texttt{australian}, $(n,d)=(690,14)$, $\lambda_1=0.5$, $\lambda_2=10^{-8}$, $\eta_1=1$, $\eta_2=10^{-6}$, $\tau=0.2$, $M=M_1$, $\gamma=2\times 10^{-7}$.}
    \label{fig: logistic australian}
\end{figure}

\subsection{Sum-of-nonconvex Example}
Similar to \cite{allen2016improved}, we generate a sum-of-nonconvex example by the following procedure:

We take $n$ normalized random vector $a_i\in\Rd$, and also $d$ vectors of the form $g_i=(0,...0,5i,0,...0)$, where the nonzero element is at $i$th coordinate.

And the sum-of-nonconvex problem is given by
\begin{equation}
\label{equ: sun of nonconvex}
    \mathop{\mathrm{minimize}}_{x\in\Rd} ~\frac{1}{2n}\sum_{i=1}^n x^T(c_ic_i^T+D_i)x + b^Tx + \lambda_1\|x\|_1,
\end{equation}
where $n=2000, d=100$, and $\lambda_1=10^{-3}$.
\begin{align*}
    c_i=\begin{cases}
    a_i+g_i \quad i=1,2,...,d,\\
    a_i \quad \quad\quad\text{otherwise}.
    \end{cases}
\end{align*}
\begin{align*}
    D_i=\begin{cases}
    -100 I \quad i=1,2,...,\frac{n}{2},\\
    100 I \quad \quad\text{otherwise}.
    \end{cases}
\end{align*}

Since the sum of $D_i$'s is $0$, they do not affect the condition number of the whole problem. However, it makes most of the first half of $f_i$ to be highly nonconvex. Overall, the condition number of this problem is equal to that of $\sum_{i=1}^n c_i c_i^T$, which is approximately 10000 in our tested data.

Since $\sum_{i=1}^n c_i c_i^T$ is diagonally dominant, we select $M=\text{diag}(\frac{1}{n}\sum_{i=1}^n c_i c_i^T)+\alpha I$ as the preconditioner. Our algorithms also have significant acceleration in this sum-of-nonconvex setting.

\begin{figure}[H]
    \centering
    \includegraphics[width=.5\textwidth]{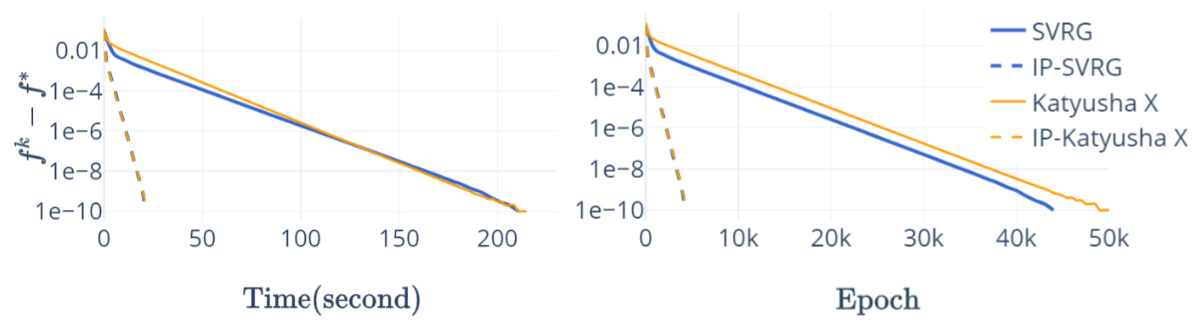}
    \caption{Sum-of-nonconvex on synthetic data. $\lambda_1=10^{-3}$,  $\alpha=15$. $\eta_1=0.015$, $\eta_2=10^{-4}$, $\tau=0.45$.}
    \label{fig: pca}
\end{figure}

\section{Conclusions and Future Work}
\label{sec: conclusions and future work}

In this paper, we propose to accelerate SVRG and Katyusha X by inexact preconditioning, with an appropriate preconditioner, both can be provably accelerated in terms of iteration complexity and gradient complexity. Our algorithms admits a nondifferentiable regularizer, as well as nonconvexity of individual functions. We confirm our theoretical results on Lasso, Logistic regression, and a sum-of-nonconvex example, where simple choices of preconditioners lead to significant accelerations.

There are still open questions left for us to address in the future: (a) Do we have theoretical guarantee when the subproblem iterator $S$ is chosen as faster schemes such as 
APCG \cite{lin2014accelerated}, NU\_ACDM \cite{allen2016even}, and A2BCD \cite{hannah2018texttt}? (b) In general, how to choose a simple preconditioner that can greatly reduce the condition number of the problem? (c) Is it possible to apply this inexact preconditioning technique to other stochastic algorithms?

\section*{Acknowledgements}
We would like to thank Yunbei Xu for helpful discussions on the idea of inexact preconditioning. We also thank the reviewers for their valuable comments.

This work is supported in part by the National Key R$\&$D Program of China 2017YFB02029, AFOSR MURI FA9550-18-1-0502, NSF DMS-1720237, and ONR N0001417121.

\bibliography{reference}
\bibliographystyle{icml2019}

\newpage
\twocolumn
\appendix

\section{Proof of Lemma \ref{lem: finite loops of FISTA with restart}}
\label{app: subproblem iterators}

In this section, we prove the results on the error generated when solving the subproblem \eqref{equ: subproblem} inexactly by Procedure \ref{alg: fixed number of inner loops}. Before proving Lemma \ref{lem: finite loops of FISTA with restart}, we will first prove a simpler case in Lemma \ref{lem: fixed number of GD}, where the subproblem iterator $S$ is the proximal gradient step.

\begin{lem}
\label{lem: fixed number of GD}
Take Assumption \ref{Assumption 1}. Suppose in Procedure \ref{alg: fixed number of inner loops}, we choose $S$ as the proximal gradient step with step size $\gamma=\eta\frac{\lambda_{\mathrm{min}}(M)}{\lambda_{\mathrm{max}}^2(M)}$, and is repeat it $p$ times, where $p \geq 1$. Then, $w_{t+1}=w_{t+1}^p$ is an approximate solution to \eqref{equ: subproblem} that satisfies
\begin{align}
\vz \in &\partial \psi(w_{t+1})+\frac{1}{\eta}M(w_{t+1}-w_t)+\tilde{\nabla}_t+M\varepsilon^p_{t+1},\label{equ: inexact optimality condition 1}\\
&\|\varepsilon^p_{t+1}\|_M\leq\frac{c(p)}{\eta}\|w_{t+1}-w_t\|_M,\label{equ: error bound 1}
\end{align}
where 
\begin{align*}
    c(p)=(\kappa(M)+1)\kappa(M)\frac{\tau^p+\tau^{p-1}}{1-\tau^p},
\end{align*}
and $\tau=\sqrt{1-\kappa^{-2}(M)}<1$. 
\end{lem}

\begin{proof}[Proof of Lemma \ref{lem: fixed number of GD}]
The optimization problem in \eqref{equ: subproblem} is of the form
\begin{align}
\label{equ: y subproblem}
\mathop{\mathrm{minimize}}_{y\in\Rd} h_1(y)+h_2(y),
\end{align}
for
$    h_1(y)=\psi(y)$ and 
$    h_2(y)=\frac{1}{2\eta}\|y-w_t\|^2_{M}+\langle\tilde{\nabla}, y\rangle.$
With our choice of $S$ as the proximal gradient descent step, the iterations in Procedure \ref{alg: fixed number of inner loops} are
\begin{align}
    w^{0}_{t+1}&=w_{t},\nonumber\\
    w^{i+1}_{t+1} &= \prox_{\gamma h_1}\big(w^{i}_{t+1}- \gamma \nabla h_2(w^i_{t+1})\big),\nonumber\\
    w_{t+1}&=w_{t+1}^p,\nonumber
\end{align}
where $i=0, 1,...,p-1$. From the definition of $\prox_{\gamma h_1}$, we have
\[
\vz\in \partial h_1(w_{t+1}^p)+\nabla h_2(w_{t+1}^{p-1})+\frac{1}{\gamma}(w_{t+1}^{p}-w_{t+1}^{p-1}).
\]
Compare this with \eqref{equ: inexact optimality condition 1} gives
\[
M\varepsilon^p_{t+1}=\frac{1}{\gamma}(w_{t+1}^p-w_{t+1}^{p-1})+\nabla h_2(w_{t+1}^{p-1})-\nabla h_2(w_{t+1}^{p}).
\]
To bound the right hand side, let $w_{t+1}^{\star}$ be the solution of \eqref{equ: y subproblem}, $\alpha=\frac{\lambda_{\text{min}}(M)}{\eta}$, and $\beta=\frac{\lambda_{\text{max}}(M)}{\eta}$. Then $h_1(y)$ is convex and $h_2(y)$ is $\alpha$-strongly convex and $\beta$-Lipschitz differentiable. Consequently, Prop. 26.16(ii) of \cite{bauschke2017convex} gives
\[
\|w_{t+1}^{i}-w_{t+1}^{\star}\|\leq \tau^{i} \|w_{t+1}^0-w_{t+1}^{\star}\|, \quad \forall i= 0, 1, ..., p,
\]
where $\tau=\sqrt{1-\gamma(2\alpha-\gamma \beta^2)}$.

Let $a_i = \|w_{t+1}^i-w_{t+1}^{\star}\|$. Then, $a_i\leq \tau^i a_0$. We can derive
\begin{align*}
\|M\varepsilon^p_{t+1}\|&\leq (\frac{1}{\gamma}+\beta)\|w_{t+1}^p-w_{t+1}^{p-1}\|\nonumber\\
&\leq (\frac{1}{\gamma}+\beta)(a_p+a_{p-1})\leq (\frac{1}{\gamma}+\beta)(\tau^p+\tau^{p-1})a_0.\nonumber
\end{align*}
On the other hand, we have
\begin{align}
\|w_{t+1}-w_t\|&\geq a_0-a_p\geq (1-\tau^p)a_0.\nonumber
\end{align}
Combining these two equations yields
\begin{align}
\|M\varepsilon^p_{t+1}\|\leq b(p)\|w_{t+1}-w_t\|,
\label{equ: proxgrad}
\end{align}
where 
\begin{align}
    b(p)= (\frac{1}{\gamma}+\frac{\lambda_{\mathrm{max}}(M)}{\eta})\frac{\tau^p+\tau^{p-1}}{1-\tau^p}.\label{equ: previous c(p)}
\end{align}
Finally, let the eigenvalues of $M$ be $0<\lambda_1\leq \lambda_2\leq...\leq \lambda_d$, with orthonormal eigenvectors $v_1, v_2,..., v_d$. Let $\varepsilon^p_{t+1}$ and $w_{t+1}-w_t$ be decomposed by
\begin{align*}
    \varepsilon^p_{t+1} &=\sum_{i=1}^d \alpha_i v_i,\\
    w_{t+1}-w_t &= \sum_{i=1}^d \beta_i v_i.
\end{align*}
then 
\begin{align*}
\|\varepsilon^p_{t+1}\|_M&=\sqrt{\sum_{i=1}^d \lambda_i\alpha_i^2}\leq \sqrt{\frac{1}{\lambda_{\mathrm{min}}(M)}\sum_{i=1}^d \lambda_i^2\alpha_i^2}\\
&=\sqrt{\frac{1}{\lambda_{\mathrm{min}}(M)}}\|M\varepsilon^p_{t+1}\|,\\
\|w_{t+1}-w_t\|&=\sqrt{\sum_{i=1}^d \beta_i^2}\leq \sqrt{\frac{1}{\lambda_{\mathrm{min}}(M)}\sum_{i=1}^d \lambda_i\beta_i^2}\\
&=\sqrt{\frac{1}{\lambda_{\mathrm{min}}(M)}}\|w_{t+1}-w_t\|_M.
\end{align*}
Combine these two inequalities with \eqref{equ: proxgrad}, we arrive at 
\begin{align}
\label{equ: proxgrad error estimate}
\|\varepsilon^p_{t+1}\|_M\leq c(p)\|w_{t+1}-w_t\|_M,
\end{align}
where 
\[
c(p)=\frac{1}{\lambda_{\mathrm{min}}(M)}b(p)=\frac{\frac{1}{\gamma}+\frac{\lambda_{\mathrm{max}}(M)}{\eta}}{\lambda_{\mathrm{min}}(M)}\frac{\tau^p+\tau^{p-1}}{1-\tau^p}.
\]
\end{proof}


Now, we are ready to prove Lemma \ref{lem: finite loops of FISTA with restart}, the techniques are similar to the proof of Lemma \ref{lem: fixed number of GD}.

\begin{proof}[Proof of Lemma \ref{lem: finite loops of FISTA with restart}]
We want to find $c(p)$ such that
\begin{align}
\vz \in &\partial \psi(w_{t+1})+\frac{1}{\eta}M(w_{t+1}-w_t)+\tilde{\nabla}_t+M\varepsilon^p_{t+1},\label{equ: inexact}\\
\|\varepsilon^p_{t+1}\|_M&\leq \frac{c(p)}{\eta}\|w_{t+1}-w_t\|_M,\label{equ: lem 4 0}
\end{align}
Take $i=r-1$ and $j=p_0-1$, then the optimality condition of the problem in line 5 of Algorithm \ref{alg: fixed number of FISTA with restart} is
\[
\vz\in \partial \psi(w_{t+1}^{(r-1, p_0)})+\frac{1}{\gamma}(w_{t+1}^{(r-1, p_0)}-u_{t+1}^{(r-1, p_0)})+\nabla h_2(u_{t+1}^{(r-1, p_0)}),
\]
compare this with \eqref{equ: inexact}, we have
\begin{align*}
M\varepsilon^p_{t+1}=&\frac{1}{\gamma}(w_{t+1}^{(r-1, p_0)}-u_{t+1}^{(r-1, p_0)})+\nabla h_2(u_{t+1}^{(r-1, p_0)})\\
& -\frac{1}{\eta}M(w_{t+1}-w_t)-\tilde{\nabla}_t\\
=& \frac{1}{\gamma}(w_{t+1}^{(r-1, p_0)}-u_{t+1}^{(r-1, p_0)})\\
&+\frac{1}{\eta}M(u^{(r-1, p_0)}_{t+1}-w_{t+1})
\end{align*}
where
\begin{align*}
u_{t+1}^{(r-1, p_0)}=&w_{t+1}^{(r-1, p_0-1)}+\frac{\theta_{p_0-2}-1}{\theta_{p_0-1}}(w_{t+1}^{(r-1, p_0-1)}-w_{t+1}^{(r-1, p_0-2)}).
\end{align*}
As a result, 
\begin{align}
    \|M\varepsilon^p_{t+1}\|\leq & \|\frac{1}{\gamma}(w_{t+1}^{(r-1, p_0)}-u_{t+1}^{(r-1, p_0)})\|\\
    &+\|\frac{1}{\eta}M(u^{(r-1, p_0)}_{t+1}-w_{t+1})\|\nonumber\\
\leq &\|\frac{1}{\gamma}(w_{t+1}^{(r-1, p_0)}-w_{t+1}^{(r-1, p_0-1)}\|\nonumber\\
&+\frac{1}{\gamma}\|\frac{\theta_{p_0-2}-1}{\theta_{p_0-1}}(w_{t+1}^{(r-1, p_0-1)}-w_{t+1}^{(r-1, p_0-2)})\|\nonumber\\
&+\|\frac{1}{\eta}M(w_{t+1}^{(r-1, p_0-1)}-w_{t+1})\|\nonumber\\
&+\|\frac{1}{\eta}\frac{\theta_{p_0-2}-1}{\theta_{p_0-1}}M(w_{t+1}^{(r-1, p_0-1)}-w_{t+1}^{(r-1, p_0-2)})\|,\label{equ: 1} 
\end{align}
Let the solution of \eqref{equ: subproblem} be $w^{\star}_{t+1}$. By Theorem 4.4 of \cite{beck2009fast}, for any $0\leq i \leq r-1$ and $0\leq j\leq p_0$ we have
\[
\Psi(w^{(i, j)}_{t+1})-\Psi(w^{\star}_{t+1})\leq \frac{2{\lambda_{\mathrm{max}}(M)}\|w^{(i,0)}_{t+1}-w^{\star}_{t+1}\|^2}{\eta j^2}.
\]
On the other hand, the strong convexity of $\Psi=h_1+h_2$ gives
\[
\Psi(w^{(i, j)}_{t+1})-\Psi(w^{\star}_{t+1})\geq \frac{{\lambda_{\mathrm{min}}(M)}}{2\eta}\|w^{(i, j)}_{t+1}-w^{\star}_{t+1}\|^2.
\]
Therefore, 
\begin{align}
\label{equ: 2}
\|w^{(i, j)}_{t+1}-w^{\star}_{t+1}\|\leq \sqrt{\frac{4\kappa(M)}{j^2}}\|w^{(i,0)}_{t+1}-w^{\star}_{t+1}\|.
\end{align}
Now, let us use \eqref{equ: 2} repeatedly to bound the right hand side of \eqref{equ: 1}. For example, the first term can be bounded as
\begin{align*}
&\|\frac{1}{\gamma}(w_{t+1}^{(r-1, p_0)}- w_{t+1}^{(r-1, p_0-1)}\|\\
\leq& \frac{1}{\gamma}\|w_{t+1}^{(r-1, p_0)}-w_{t+1}^{\star}\|\\
&+\frac{1}{\gamma}\|w_{t+1}^{(r-1, p_0-1)}-w_{t+1}^{\star}\|\nonumber\\
\leq& \frac{1}{\gamma}(\frac{4\kappa(M)}{p_0^2})^{\frac{r}{2}}\|w_{t+1}^{(0,0)}-w^{\star}_{t+1}\|\\
&+\frac{1}{\gamma}(\frac{4\kappa(M)}{p_0^2})^{\frac{r-1}{2}}(\frac{4\kappa(M)}{(p_0-1)^2})^{\frac{1}{2}}\|w_{t+1}^{(0,0)}-w^{\star}_{t+1}\|.
\end{align*}
Similarly, the rest of the terms can be bounded as follows,
\begin{align*}
&\frac{1}{\gamma}\|\frac{\theta_{p_0-2}-1}{\theta_{p_0-1}} (w_{t+1}^{(r-1, p_0-1)}-w_{t+1}^{(r-1, p_0-2)})\|\nonumber\\
\leq& \frac{1}{\gamma}(\frac{4\kappa(M)}{p_0^2})^{\frac{r-1}{2}}(\frac{4\kappa(M)}{(p_0-1)^2})^{\frac{1}{2}}\|w_{t+1}^{(0,0)}-w^{\star}_{t+1}\|\\
&+\frac{1}{\gamma}(\frac{4\kappa(M)}{p_0^2})^{\frac{r-1}{2}}(\frac{4\kappa(M)}{(p_0-2)^2})^{\frac{1}{2}}\|w_{t+1}^{(0,0)}-w^{\star}_{t+1}\|,
\end{align*}
\begin{align*}
&\|\frac{1}{\eta}M(w_{t+1}^{(r-1, p_0-1)}-w_{t+1})\|\\
\leq&\frac{\lambda_{\mathrm{max}}(M)}{\eta}(\frac{4\kappa(M)}{p_0^2})^{\frac{r-1}{2}}(\frac{4\kappa(M)}{(p_0-1)^2})^{\frac{1}{2}}\|w_{t+1}^{(0,0)}-w^{\star}_{t+1}\|,\nonumber\\
&+\frac{\lambda_{\mathrm{max}}(M)}{\eta}(\frac{4\kappa(M)}{p_0^2})^{\frac{r}{2}}\|w_{t+1}^{(0,0)}-w^{\star}_{t+1}\|,\nonumber
\end{align*}
\begin{align*}
&\|\frac{1}{\eta}\frac{\theta_{p_0-2}-1}{\theta_{p_0-1}}M(w_{t+1}^{(r-1, p_0-1)}-w_{t+1}^{(r-1, p_0-2)})\|\nonumber\\
\leq& \frac{\lambda_{\mathrm{max}}(M)}{\eta}(\frac{4\kappa(M)}{p_0^2})^{\frac{r-1}{2}}(\frac{4\kappa(M)}{(p_0-1)^2})^{\frac{1}{2}}\|w_{t+1}^{(0,0)}-w^{\star}_{t+1}\|\\
&+\frac{\lambda_{\mathrm{max}}(M)}{\eta}(\frac{4\kappa(M)}{p_0^2})^{\frac{r-1}{2}}(\frac{4\kappa(M)}{(p_0-2)^2})^{\frac{1}{2}}\|w_{t+1}^{(0,0)}-w^{\star}_{t+1}\|,
\end{align*}
where in the first and third estimate we have used $\frac{\theta_{p_0-2}-1}{\theta_{p_0-1}}\leq \frac{\theta_{p_0-2}}{\theta_{p_0-1}}<1$. On the other hand, we have
\begin{align*}
&\|w_{t+1}-w_t\|=\|w^{(r-1,p_0)}_{t+1}-w^{(0,0)}_{t+1}\|
\\&\geq \|w^{(0,0)}_{t+1}-w^{\star}_{t+1}\|-\|w^{(r-1,p_0)}_{t+1}-w^{\star}_{t+1}\|\\
&\geq (1-(\frac{4\kappa(M)}{p_0^2})^{\frac{r}{2}})\|w^{(0,0)}_{t+1}-w^{\star}_{t+1}\|.
\end{align*}
As a result, taking $\gamma=\frac{\lambda_{\mathrm{max}}(M)}{\eta}$, $w^{(0,0)}_{t+1}=w_{t}$, $w^{(r-1, p_0)}_{t+1}=w_{t+1}$ and $\tau=(\frac{4\kappa(M)}{p_0^2})^{\frac{1}{2 p_0}}$ yields
\[
\|M\varepsilon^p_{t+1}\|\leq 2\frac{\lambda_{\mathrm{max}}(M)}{\eta}\frac{b(p)}{1-\tau^p}\|w_{t+1}-w_{t}\|,
\]
where
\begin{align}
b(p)=&\tau^{p-p_0}\big((\frac{4\kappa(M)}{(p_0-1)^2})^{\frac{1}{2}}+(\frac{4\kappa(M)}{(p_0-2)^2})^{\frac{1}{2}}\big)\nonumber\\
&+\tau^p+\tau^{p-p_0}(\frac{4\kappa(M)}{(p_0-1)^2})^{\frac{1}{2}}.\label{equ: lem 4 11}
\end{align}
Similar to the end of proof of Lemma \ref{lem: fixed number of GD}, we have
\[
\|M\varepsilon^p_{t+1}\|_M\leq 2\frac{\kappa(M)}{\eta}\frac{b(p)}{1-\tau^p}\|w_{t+1}-w_{t}\|_M.
\]
Now, let us choose $p_0$ such that $\tau=(\frac{4\kappa(M)}{p_0^2})^{\frac{1}{2 p_0}}$ is minimized, a simple calculation yields
\[
p^{\star}_0=2e\sqrt{\kappa(M)}.
\]
In order for $p_0$ to be an integer, we can take
\[
p_0=\lceil 2e\sqrt{\kappa(M)}\rceil,
\]
then 
\begin{align*}
\tau&=(\frac{4\kappa(M)}{p_0^2})^{\frac{1}{2 p_0}}\leq (\frac{1}{e^2})^{\frac{1}{2 \lceil 2e\sqrt{\kappa(M)}\rceil}}\leq (\frac{1}{e^2})^{\frac{1}{2 ( 2e\sqrt{\kappa(M)}+1)}}\\
&=\exp(-\frac{1}{2e\sqrt{\kappa(M)}+1}).
\end{align*}
Finally, Let us show that $b(p)$ in \eqref{equ: lem 4 11} can be bounded by $7\tau^p$, and the desired bound \eqref{equ: lem 4 0} on $\|\varepsilon^p_{t+1}\|_M$ follows.

First, we have
\begin{align*}
\tau^{-p_0}(\frac{4\kappa(M)}{p_0-1})^{\frac{1}{2}}=(\frac{p_0}{p_0-1})^{\frac{1}{p_0}},
\end{align*}
and 
\[
p_0=\lceil 2e\sqrt{\kappa(M)}\rceil\geq \lceil2e\rceil=6.
\]
On the other hand, a simple calculation shows that $(\frac{p_0}{p_0-1})^{\frac{1}{p_0}}$ is decreasing in $p_0$, therefore 
\[
\tau^{-p_0}(\frac{4\kappa(M)}{p_0-1})^{\frac{1}{2}}\leq (\frac{6}{5})^{\frac{1}{6}}<2,
\]
Similarly, one can show that
\[
\tau^{-p_0}(\frac{4\kappa(M)}{p_0-2})^{\frac{1}{2}}\leq (\frac{6}{4})^{\frac{1}{6}}<2.
\]
Combining these two inequalities with \eqref{equ: lem 4 1} yields $$b(p)\leq 7\tau^p.$$
\end{proof}


\section{Proof of Theorem \ref{thm: convergence of inexact preconditioned svrg}}
\label{app: convergence of inexact preconditioned svrg}
In this section, we proceed to establish the convergence of inexact preconditioned SVRG as in Algorithm \ref{alg: inexact preconditioned svrg}. The proof is similar to that of Theorem D.1 of \cite{allen2018katyusha}.

Before proving Theorem \ref{thm: convergence of inexact preconditioned svrg}, let us first prove several lemmas.

First, the inexact optimality condition \eqref{equ: inexact optimality condition} gives the following descent:
\begin{lem}
\label{lem: inexact descent}
Under Assumption \ref{Assumption 1}, suppose that \eqref{equ: inexact optimality condition} holds. Then, for any $u\in\Rd$ we have 
\begin{align*}
    &\langle\tilde{\nabla}_t, w_t-u \rangle+\psi(w_{t+1})-\psi(u)\\
    &\leq \langle\tilde{\nabla}_t, w_t-w_{t+1} \rangle +\frac{\|u-w_t\|_M^2}{2\eta}\\
    &-\frac{1}{2\eta}\|u-w_{t+1}\|_M^2-\frac{1}{2\eta}\|w_{t+1}-w_{t}\|_M^2\\
    &+\langle M\varepsilon^p_{t+1}, u-w_{t+1}\rangle.
\end{align*}
\end{lem}

\begin{proof}
First, let us rewrite the left hand side as
\begin{align*}
&\langle\tilde{\nabla}_t, w_t-u \rangle+\psi(w_{t+1})-\psi(u)\\
&=\langle\tilde{\nabla}_t, w_t-w_{t+1} \rangle+\langle\tilde{\nabla}_t, w_{t+1}-u \rangle+\psi(w_{t+1})-\psi(u).
\end{align*}
By \eqref{equ: inexact optimality condition} and the definition of subdifferential we have
\[
\psi(u)\geq \psi(w_{t+1})-\langle \tilde{\nabla}_t+\frac{1}{\eta}M(w_{t+1}-w_{t})+M\varepsilon^p_{t+1}, u-w_{t+1}\rangle.
\]
Combining these two gives
\begin{align*}
    &\langle\tilde{\nabla}_t, w_t-u \rangle+\psi(w_{t+1})-\psi(u)\\
    \leq& \langle\tilde{\nabla}_t, w_t-w_{t+1} \rangle\\
    &+\langle \frac{1}{\eta}M(w_{t+1}-w_t)+M\varepsilon^p_{t+1}, u-w_{t+1}\rangle\\
    =& \langle\tilde{\nabla}_t, w_t-w_{t+1} \rangle +\frac{\|u-w_t\|_M^2}{2\eta}\\
    &-\frac{1}{2\eta}\|u-w_{t+1}\|_M^2-\frac{1}{2\eta}\|w_{t+1}-w_{t}\|_M^2\\
    &+\langle M\varepsilon^p_{t+1}, u-w_{t+1}\rangle,
\end{align*}
where in the last equality we have applied
\[
\langle a-b, c-a \rangle_{M}=-\frac{1}{2}\|a-b\|^2_{M}-\frac{1}{2}\|a-c\|^2_{M}+\frac{1}{2}\|b-c\|_{M}.
\]
\end{proof}

Based on lemma \ref{lem: inexact descent}, we have
\begin{lem}
\label{lem: lem 1 for thm 1}
Under Assumption \ref{Assumption 1}, if the iterator $S$ in Procedure \ref{alg: fixed number of inner loops} is proximal gradient descent or FISTA with restart, then, for any $a>0$, $\eta\leq\frac{1-2c(p)a}{2L^M_f}$, and $u\in\Rd$ we have
\begin{align*}
&\mathbb{E}[F(w_{t+1})-F(u)]\\
\leq&  \mathbb{E}[ \eta\|\tilde{\nabla}_t-\nabla f(w_t)\|^2_{M^{-1}}+\frac{1-\eta\sigma^M_f}{2\eta}\|u-w_t\|_M^2\\
&-(\frac{1}{2\eta}-\frac{c(p)}{2\eta a})\|u-w_{t+1}\|_M^2].
\end{align*}
\end{lem}
\begin{proof}
We have
\begin{align}
&\mathbb { E } [ F \left( w _ { t + 1 } \right) - F ( u ) ]\nonumber\\
&= \mathbb { E } [ f \left( w _ { t + 1 } \right) - f ( u ) + \psi \left( w _ { t + 1 } \right) - \psi ( u ) ]\nonumber\\
\leq&\mathbb { E } [ f \left( w _ { t } \right) + \left\langle \nabla f \left( w _ { t } \right) , w _ { t + 1 } - w _ { t } \right\rangle  \nonumber\\
&+ \frac { L^M_f } { 2 } \left\| w _ { t } - w _ { t + 1 } \right\|_M ^ { 2 }- f ( u ) + \psi \left( w _ { t + 1 } \right) - \psi ( u ) ]  \nonumber\\  
&\mathbb { E } [ \left\langle \nabla f \left( w _ { t } \right) , w _ { t } - u \right\rangle - \frac { \sigma _ { f }^M } { 2 } \left\| u - w _ { t } \right\|_M ^ { 2 } \nonumber\\ 
&+\left\langle \nabla f \left( w _ { t } \right) , w _ { t + 1 } - w _ { t } \right\rangle + \frac { L^M_f } { 2 } \left\| w _ { t } - w _ { t + 1 } \right\|_M ^ { 2 }\nonumber\\
&+ \psi \left( w _ { t + 1 } \right) - \psi ( u ) ]  \nonumber\\
=&\mathbb { E } [ \langle \tilde { \nabla } _ { t } , w _ { t } - u \rangle - \frac { \sigma^M _ { f } } { 2 } \left\| u - w _ { t } \right\|_M ^ { 2 }\\
&+ \left\langle \nabla f \left( w _ { t } \right) , w _ { t + 1 } - w _ { t } \right\rangle\nonumber\\
&+ \frac { L^M_f } { 2 } \left\| w _ { t } - w _ { t + 1 } \right\|_M ^ { 2 } + \psi \left( w _ { t + 1 } \right) - \psi ( u ) ],  \label{equ: lem 4 1}
\end{align}
where the first and second inequality are due to the strong convexity and smoothness under $\|\cdot\|_M$ in Assumption \ref{Assumption 1}, respectively. the last equality is due to $\E[\tilde{\nabla}_t]=\nabla f(w_t)$.

On the other hand, recall that Lemma \ref{lem: inexact descent} gives
\begin{align*}
    &\langle\tilde{\nabla}_t, w_t-u \rangle+\psi(w_{t+1})-\psi(u)\\
    & \langle\tilde{\nabla}_t, w_t-w_{t+1} \rangle +\frac{\|u-w_t\|_M^2}{2\eta}\\
    &-\frac{1}{2\eta}\|u-w_{t+1}\|_M^2-\frac{1}{2\eta}\|w_{t+1}-w_{t}\|_M^2\\
    &+\langle M\varepsilon^p_{t+1}, u-w_{t+1}\rangle,
\end{align*}
For the last term we can apply Cauchy-Schwartz as follows,
\[
\langle M\varepsilon^p_{t+1}, u-w_{t+1}\rangle\leq \|\varepsilon^p_{t+1}\|_M\|u-w_{t+1}\|_M,
\]
from Lemma \ref{lem: fixed number of GD} and Lemma \ref{lem: finite loops of FISTA with restart} we know that
\[
\|\varepsilon^p_{t+1}\|_M\leq \frac{c(p)}{\eta}\|w_{t+1}-w_t\|_M.
\]
Therefore, by Young's inequality, we have for any $a>0$ that
\begin{align*}
&\langle M\varepsilon^p_{t+1}, u-w_{t+1}\rangle\\
\leq&\frac{c(p)a}{2\eta}\|w_{t+1}-w_t\|^2_M+\frac{c(p)}{2a\eta}\|u-w_{t+1}\|^2_M.
\end{align*}
Applying this to Lemma \ref{lem: inexact descent} yields
\begin{align*}
    &\langle\tilde{\nabla}_t, w_t-u \rangle+\psi(w_{t+1})-\psi(u)\\
    \leq&\langle\tilde{\nabla}_t, w_t-w_{t+1} \rangle +\frac{\|u-w_t\|_M^2}{2\eta}\\
    &-\frac{1}{2\eta}\|u-w_{t+1}\|_M^2-\frac{1}{2\eta}\|w_{t+1}-w_{t}\|_M^2\\
    &+\langle M\varepsilon^p_{t+1}, u-w_{t+1}\rangle\\
    & \langle\tilde{\nabla}_t, w_t-w_{t+1} \rangle +\frac{\|u-w_t\|_M^2}{2\eta}\\
    &-(\frac{1}{2\eta}-\frac{c(p)}{2a\eta})\|u-w_{t+1}\|_M^2\\
    &-(\frac{1}{2\eta}-\frac{c(p)a}{2\eta})\|w_{t+1}-w_{t}\|_M^2
\end{align*}
Applying this to \eqref{equ: lem 4 1}, we arrive at
\begin{align*}
& \mathbb { E } [ F \left( w _ { t + 1 } \right) - F ( u ) ]  \\ 
\leq& \mathbb { E } [ \langle\tilde { \nabla } _ { t } - \nabla f \left( w _ { t } \right) , w _ { t } - w _ { t + 1 } \rangle \\
&- \frac { 1 -c(p)a-\eta L^M_f } { 2 \eta } \left\| w _ { t } - w _ { t + 1 } \right\|_M ^ { 2 }\\
& +  \frac{1 - \eta\sigma^M _ { f }}{2\eta} \left\| u - w _ { t } \right\|_M ^ { 2 }-(\frac{1}{2\eta}-\frac{c(p)}{2a\eta})\|u-w_{t+1}\|_M^2 ]\\
& \mathbb { E } [ \frac{\eta}{2(1-c(p)a-\eta L^M_f)}\|\tilde{\nabla}_t-\nabla f(w_t)\|_{M^{-1}}^2\\
& +  \frac{1 - \eta\sigma^M _ { f }}{2\eta} \left\| u - w _ { t } \right\|_M ^ { 2 }-(\frac{1}{2\eta}-\frac{c(p)}{2a\eta})\|u-w_{t+1}\|_M^2 ],
\end{align*}
where in the second inequality we have applied 
\begin{align*}
&\langle u_1, u_2 \rangle=\langle M^{-\frac{1}{2}}u_1, M^{\frac{1}{2}}u_2\rangle\leq\|u_1\|_{M^{-1}}\|u_2\|_{M}\\
& \leq\frac{1}{2b}\|u_1\|^2_{M_1^{-1}}+\frac{b}{2}\|u_2\|^2_{M^{\frac{1}{2}}}\quad\text{for any $b>0$}.
\end{align*}

Finally, since $\eta\leq\frac{1-2c(p)a}{2L^M_f}$, we have
$\frac{\eta}{2(1-c(p)a-\eta L^M_f)} \leq \eta$, which gives the desired result.

\end{proof}
\begin{lem}
\label{lem: lem 2 for thm 1}
Under Assumption \ref{Assumption 1}, we have
\[
\E[\|\tilde{\nabla}_t-\nabla f (w_t)\|_{M^{-1}}^2]\leq  (L^M_f)^2\|w_0-w_t\|^2_M.
\]
\end{lem}
\begin{proof}
We have
\begin{align*}
&\E[\|\tilde{\nabla}_t-\nabla f (w_t)\|_{M^{-1}}^2]\\
=&\E[\|\nabla f(w_0)+\nabla f_{i_t}(w_t)-\nabla f_{i_t}(w_0)-\nabla f(w_t)\|^2_{M^{-1}}]\\
=&\E[\|\big(\nabla f_{i_t}(w_t)-\nabla f_{i_t}(w_0)\big)-\big(\nabla f(w_t)-\nabla f(w_0)\big)\|_{M^{-1}}^2]\\
\leq& \E[\|\nabla f_{i_t}(w_t)-\nabla f_{i_t}(w_0)\|_{M^{-1}}^2]\\
\leq& (L^M_f)^2\|w_t-w_0\|_M^2,
\end{align*}
where in the first inequality, we have applied $\E[\|\xi-\E\xi\|^2]=\E[\|\xi\|^2-\|\E\xi\|^2$ with $\xi=M^{-\frac{1}{2}}\big(\nabla f_{i_t}(w_t)-\nabla f_{i_t}(w_0)\big)$, and in the second inequality follows from Assumption \ref{Assumption 1}.
\end{proof}

\begin{lem}[(Fact 2.3 of \cite{allen2018katyusha})]
\label{lem: lem 3 for thm 1}
Let $C_1, C_2,...$ be a sequence of numbers, and $N \sim $\textbf{Geom}$(p)$, then
\begin{enumerate}
   \item $\mathbb { E } _ { N } \left[ C _ { N } - C _ { N + 1 } \right] = \frac { p } { 1 - p } \mathbb { E }_N \left[ C _ { 0 } - C _ { N } \right]$, and      
   \item $\mathbb { E } _ { N } \left[ C _ { N } \right] = ( 1 - p ) \mathbb { E } \left[ C _ { N + 1 } \right] + p C _ { 0 }.$  
\end{enumerate}
\end{lem}

\begin{lem}
\label{lem: lem 4 for thm 1}
Under Assumption \ref{Assumption 1}, if $\eta\leq\min\{\frac{1-2c(p)a}{2L^M_f}, \frac{1}{2\sqrt{m}L^M_f}\}$ and $m\geq 2$, then, for any $u\in\Rd$ we have 
\begin{align*}
&\E[F(w_{D+1})-F(u)]\\
\leq&\E[-\frac{1}{4m\eta}\|w_{D+1}-w_0\|_M^2+\frac{\langle w_0-w_{D+1}, w_0-u\rangle_M}{m\eta}\\
&-(\frac{\sigma^M_f}{4}-\frac{c(p)}{2a\eta})\|w_{D+1}-u\|^2_M].
\end{align*}
\end{lem}
\begin{proof}
By Lemmas \ref{lem: lem 1 for thm 1} and \ref{lem: lem 2 for thm 1}, we know that 
\begin{align*}
&\mathbb{E}[F(w_{t+1})-F(u)]\\
&  \mathbb{E}[ \eta (L^M_f)^2\|w_0-w_t\|^2_{M}+\frac{1-\eta\sigma^M_f}{2\eta}\|u-w_t\|_M^2\\
&-(\frac{1}{2\eta}-\frac{c(p)}{2\eta a})\|u-w_{t+1}\|_M^2].
\end{align*}
Let $D \sim$ \textbf{Geom}$(\frac{1}{m})$ as in Algorithm \ref{alg: inexact preconditioned svrg} and take $t=D$, then 
\begin{align*}
&\mathbb{E}[F(w_{D+1})-F(u)]\\
\leq&  \mathbb{E}[ \eta (L^M_f)^2\|w_0-w_D\|^2_{M}+\frac{1}{2\eta}\|u-w_D\|_M^2\\
&-\frac{1}{2\eta}\|u-w_{D+1}\|_M^2-\frac{\sigma^M_f}{2}\|u-w_D\|^2_M\\
&+\frac{c(p)}{2\eta a}\|u-w_{D+1}\|_M^2]\\
=&  \mathbb { E } [ \eta (L^M_f)^2 \| w _ { D } - w _ { 0 } \|_M ^ { 2 } + \frac { \| u - w _ { 0 } \|_M ^ { 2 } - \| u - w _ { D } \|_M ^ { 2 } } { 2 ( m - 1 ) \eta } \\
&- \frac { \sigma_f^M } { 2 } \| u - w _ { D } \|_M ^ { 2 } +\frac{c(p)}{2a\eta}\|u-w_{D+1}\|_M^2]  \\  
=&\mathbb { E } [ \frac { m - 1 } { m } \eta (L_f^M)^2 \| w _ { D + 1 } - w _ { 0 } \|_M ^ { 2 }\\
&+ \frac { \| u - w _ { 0 } \|_M ^ { 2 } - \| u - w _ { D + 1 } \|_M ^ { 2 } } { 2 m \eta } ]\\  
& - \frac { \sigma_f^M } { 2 m } \| u - w _ { 0 } \|_M ^ { 2 } - \frac { \sigma_f^M ( m - 1 ) } { 2 m } \| u - w _ { D + 1 } \|_M ^ { 2 }\\
&+ \frac{c(p)}{2a\eta}\|u-w_{D+1}\|_M^2]  \\ 
\leq&\mathbb { E } [ \eta (L_f^M)^2 \| w _ { D + 1 } - w _ { 0 } \|_M ^ { 2 } + \frac { \| u - w _ { 0 } \|_M ^ { 2 } - \| u - w _ { D + 1 } \|_M ^ { 2 } } { 2 m \eta } \\
&- \frac { \sigma_f^M } { 4 } \| u - w _ { D + 1 } \|_M ^ { 2 }+ \frac{c(p)}{2a\eta}\|u-w_{D+1}\|_M^2]  \\  
\leq&\mathbb { E } [ - \frac { 1 } { 4 m \eta } \left\| w _ { 0 } - w _ { D + 1 } \right\|_M ^ { 2 } \\
&+ \frac { \left\| u - w _ { 0 } \right\|_M ^ { 2 } - \left\| u - w _ { D + 1 } \right\|_M ^ { 2 } + \left\| w _ { 0 } - w _ { D + 1 } \right\|_M ^ { 2 } } { 2 m \eta }\\
&- \frac { \sigma^M _ { f }} { 4 } \left\| w _ { D + 1 } - u \right\|_M ^ { 2 }+ \frac{c(p)}{2a\eta}\|u-w_{D+1}\|_M^2]  \\  
=&\mathbb { E } [ - \frac { 1 } { 4 m \eta } \| w _ { D + 1 } - w _ { 0 } \|_M ^ { 2 } + \frac { \langle w _ { 0 } - w _ { D + 1 } , w _ { 0 } - u \rangle_M } { m \eta }\\
&- (\frac { \sigma_f^M } { 4 }-\frac{c(p)}{2a\eta}) \| w _ { D + 1 } - u \|_M ^ { 2 } ],
\end{align*}
where the first equality follows from the item 1 of Lemma \ref{lem: lem 3 for thm 1} with $C_N=\|u-w_N\|_M^2$, the second inequality follows from item 2 with $C_N=\|w_d-w_0\|_M^2$, item 2 with $C_N=\|u-w_0\|_M^2-\|u-w_N\|_M^2$, and item 1 with $C_N=\|u-w_D\|_M^2$, then third inequality makes use of $m\geq 2$ and the fourth inequality makes use of $\eta\leq \frac{1}{2\sqrt{m}L^M_f}$.

\end{proof}
Now, let us proceed to prove Theorem \ref{thm: convergence of inexact preconditioned svrg}. With Lemma \ref{lem: lem 4 for thm 1}, it can be proved in a similar way as Theorem 3 of \cite{hannah2018breaking}.
\begin{proof}[Proof of Theorem \ref{thm: convergence of inexact preconditioned svrg}]
Without loss of generality, we can assume $x^{\star}=\argmin_{x\in\Rd} F(x)=\mathbf{0}$ and $F(x^*)=0.$

According to Lemma \ref{lem: lem 4 for thm 1}, for any $u\in\R^d$, and $\eta\leq\min\{\frac{1-2c(p)a}{2L^M_f}, \frac{1}{2\sqrt{m}L^M_f}\}$ we have
\begin{align*}
&\mathbb{E}[F(x^{j+1})-F(u)]\\
\leq& \mathbb{E}[-\frac{1}{4m\eta}\|x^{j+1}-x^j\|_M^2\\
&+\frac{\langle x^j-x^{j+1}, x^j-u\rangle_M}{m\eta}-(\frac{\sigma^M_f}{4}-\frac{c(p)}{2a\eta})\|x^{j+1}-u\|_M^2],
\end{align*}
or equivalently,
\begin{align*}
&\mathbb{E}[F(x^{j+1})-F(u)]\\
\leq& \mathbb{E}[\frac{1}{4m\eta}\|x^{j+1}-x^j\|_M^2+\frac{1}{2m\eta}\|x^j-u\|_M^2\\
&-\frac{1}{2m\eta}\|x^{j+1}-u\|_M^2-(\frac{\sigma_f^M}{4}-\frac{c(p)}{2a\eta})\|x^{j+1}-u\|_M^2].
\end{align*}
In the following proof, we will omit $\E$.

Setting $u=x^*=0$ and $u=x^j$ yields the following two inequalities:
\begin{align}
    F(x^{j+1})&\leq \frac{1}{4m\eta}(\|x^{j+1}-x^j\|_M^2+2\|x^j\|_M^2)\nonumber\\
    &-\frac{1}{2m\eta}\big(1+\frac{1}{2}m\eta(\sigma_f^M-\frac{2c(p)}{a\eta})\big)\|x^{j+1}\|_M^2,\label{1}
    \end{align}
    \begin{align}
    &F(x^{j+1})-F(x^j)\\
    \leq&-\frac{1}{4m\eta}\big(1+m\eta(\sigma_f^M-\frac{2c(p)}{a\eta})\big)\|x^{j+1}-x^j\|_M^2.\label{2}
\end{align}
Define $\tau=\frac{1}{2}m\eta(\sigma_f^M-\frac{2c(p)}{a\eta})$, multiply $(1+2\tau)$ to \eqref{1}, then add it to \eqref{2} yields
\begin{align*}
&2(1+\tau)F(x^{j+1})-F(x^j)\\
\leq& \frac{1}{2m\eta}(1+2\tau)\big(\|x^j\|_M^2-(1+\tau)\|x^{j+1}\|_M^2\big).
\end{align*}
Multiplying both sides by $(1+\tau)^j$ gives
\begin{align*}
&2(1+\tau)^{j+1}F(x^{j+1})-(1+\tau)^j F(x^j)\\
\leq& \frac{1}{2m\eta}(1+2\tau)\big((1+\tau)^j\|x^j\|_M^2-(1+\tau)^{j+1}\|x^{j+1}\|^2_M\big).
\end{align*}
Summing over $j=0,1,...,k-1$, we have
\begin{align*}
&(1+\tau)^k F(x^k)+\sum_{j=0}^{k-1}(1+\tau)^j F(x^j)-F(x^0)\\
\leq& \frac{1}{2m\eta}(1+2\tau)(\|x^0\|_M^2-(1+\tau)^k\|x^k\|_M^2).
\end{align*}
Since $F(x^j)\geq 0$, we have
\[
F({x}^k)(1+\tau)^k \leq F(x^0)+\frac{1}{2m\eta}(1+2\tau)\|x^0\|^2.
\]
By the strong convexity of $F$, we have $F(x^0)\geq \frac{\sigma_f^M}{2}\|x^0\|_M^2$, therefore
\begin{align}
\label{equ: thm 1 0}
F({x}^k)(1+\tau)^k\leq F(x^0)(2+\frac{1}{2\tau}).
\end{align}
Finally, recall that $a>0$ can be chosen arbitrarily, so we can take 
$$a=\frac{4c(p)}{\eta\sigma^M_f},$$ 
and 
\begin{align}
    \eta\leq&\min\{\frac{1-2c(p)a}{2L^M_f}, \frac{1}{2\sqrt{m}L^M_f}\}\nonumber\\
    =&\min\{\frac{1-\frac{8c^2(p)}{\eta\sigma^M_f}}{2L^M_f}, \frac{1}{2\sqrt{m}L^M_f}\},\label{equ: thm 1 1}
\end{align}
\[
\tau=\frac{1}{2}m\eta(\sigma_f^M-\frac{2c(p)}{a\eta})=\frac{1}{4}m\eta\sigma^M_f.
\]
In order for the choice of $\eta$ in \eqref{equ: thm 1 1} to be possible,
we need
\begin{align}
\label{equ: thm 1 2}
2L^M_f\eta^2-\eta+8\frac{c^2(p)}{\sigma^M_f}\leq 0
\end{align}
to have one solution at least, which requires
\[
64\kappa^M_f c^2(p)\leq 1,
\]
under which $\eta=\frac{1}{4L^M_f}$ satisfy \eqref{equ: thm 1 2}.
As a result, $m\geq 4$ makes \eqref{equ: thm 1 1} into
\[
\eta\leq \frac{1}{2\sqrt{m}L^M_f},
\]
and the desired convergence result follows from \eqref{equ: thm 1 0}.
\end{proof}

\section{Proof of Lemma \ref{lem: choice of p with FISTA with restart}}
\label{app: Proof of choice of p with FISTA with researt}
\begin{proof}
From Lemma \ref{lem: finite loops of FISTA with restart}, we know that 
\[
c(p)=14\kappa(M)\frac{\tau^p}{1-\tau^p},
\]
where
\[
\tau\leq \exp(-\frac{1}{2e\sqrt{\kappa(M)}+1}).
\]
Therefore, in order for $64\kappa^M_f c^2(p)\leq 1$, we need
\[
\kappa^M_f\kappa^2(M)(\frac{\tau^p}{1-\tau^p})^2\leq \frac{1}{64\times 14^2}=c_1,
\]
which is equivalent to
\[
\tau^p\leq \frac{c_1}{\sqrt{\kappa^M_f}\kappa(M)+\sqrt{c_1}}.
\]
Thus, it suffices to require that
\[
[\exp(-\frac{1}{2e\sqrt{\kappa(M)}+1})]^p\leq \frac{c}{\sqrt{\kappa^M_f}\kappa(M)+\sqrt{c_1}},
\]
which gives 
\[
p\geq (2e\sqrt{\kappa(M)}+1)\ln\frac{\sqrt{\kappa^M_f}\kappa(M)+\sqrt{c_1}}{c_1}.
\]
\end{proof}

\section{Proof of Theorem \ref{thm: convergence of inexact preconditioned Katyusha X}}
\label{app: convergence of inexact preconditioned Katyusha X}

The proof of Theorem \ref{thm: convergence of inexact preconditioned Katyusha X} is similar to that of Theorem 4.3 of \cite{allen2018katyusha}, so we provide a proof sketch here and omit the details.
\begin{enumerate}
    \item In \cite{allen2018katyusha}, the proof of Theorem 4.3 is based on Lemma 3.3, here the proof of Theorem \ref{thm: convergence of inexact preconditioned Katyusha X} is based on Lemma \ref{lem: lem 4 for thm 1}, which is an analog of Lemma of 3.3 in our settings.
    \item Based on Lemma \ref{lem: lem 4 for thm 1}, the proof of Theorem \ref{thm: convergence of inexact preconditioned Katyusha X} follows in nearly the same way as Theorem 4.3 of \cite{allen2018katyusha}, the only difference is that one needs to replace $\sigma$ by $\sigma^M_f-\frac{2c(p)}{a\eta}$. 
    
    \item By setting
    $$a=\frac{4c(p)}{\eta\sigma^M_f},$$ 
    and
    \[
    64\kappa^M_f c^2(p)\leq 1
    \]
    as in the proof of Theorem \ref{thm: convergence of inexact preconditioned svrg},
    the $\tau$ in Theorem 4.3 of \cite{allen2018katyusha} becomes $\frac{1}{2}m\eta\sigma^{M}_f$, and the convergence result of Theorem \ref{thm: convergence of inexact preconditioned Katyusha X} follows.
\end{enumerate}

\section{Proof of Theorems \ref{thm: inexact preconditioned svrg vs svrg} and \ref{thm: inexact preconditioned Katyusha X vs Katyusha X}}

\begin{proof}[Proof of Theorem \ref{thm: inexact preconditioned svrg vs svrg}]
From Remark \ref{rem: M=I gives previous results}, we know that the gradient complexity of SVRG can be expressed as 
\[
C_1(m, \varepsilon)=\cO(\frac{n+m}{\ln(1+\frac{1}{4}m\eta \sigma_f)}\ln{\frac{1}{\varepsilon}}).
\]
Taking the largest possible step size $\eta=\frac{1}{2\sqrt{m}L_f}$ as in Theorem \ref{thm: convergence of inexact preconditioned svrg}, we have
\[
C_1(m, \varepsilon)=\cO(\frac{n+m}{\ln(1+\frac{\sqrt{m}}{8\kappa_f})}\ln{\frac{1}{\varepsilon}}).
\]
Let us first find the optimal $m=m^{\star}$ for SVRG, let 
\[
g(m) = \frac{n+m}{\ln(1+\frac{\sqrt{m}}{8\kappa_f})},
\]
then 
\[
g'(m)=\frac{\ln(1+\frac{\sqrt{m}}{8\kappa_f})-\frac{\frac{\sqrt{m}}{8\kappa_f}}{1+\frac{\sqrt{m}}{8\kappa_f}}\frac{n+m}{2m}}{\ln^2(1+z)}.
\]
Taking derivative to the numerator gives
\begin{align*}
&[\ln(1+\frac{\sqrt{m}}{8\kappa_f})-\frac{\frac{\sqrt{m}}{8\kappa_f}}{1+\frac{\sqrt{m}}{8\kappa_f}}\frac{n+m}{2m}]'\\
&=(n+m)\frac{\frac{1}{32\kappa_f}m^{-\frac{3}{2}}+2\frac{m^{-1}}{(16\kappa_f)^2}}{(1+\frac{\sqrt{m}}{8\kappa_f})^2}>0,
\end{align*}
Therefore, $m^{\star}$ is given by $g'(m)=0$. Let $z=\frac{\sqrt{m}}{8\kappa_f}>0$, then
\[
g'(m)=\frac{\ln(1+z)-\frac{z}{1+z}\frac{n+m}{2m}}{\ln^2(1+z)}.
\]
Since $\ln(1+z)> \frac{z}{1+z}$ for $z> 0$, we know that $g'(n)>0$, therefore, $m^{\star}<n$.

Let $m=n^s$ where $0<s<1$, we would like to have $g'(n^s)<0$, i,e., 
\[
\frac{\ln(1+z)}{\frac{z}{1+z}}<\frac{1+n^{1-s}}{2}.
\]
so that $m^{\star}\in(n^s, n)$. 

Since $\kappa_f>n^{\frac{1}{2}}$, we have $z=\frac{\sqrt{m}}{8\kappa_f}<\frac{1}{8}$, on the other hand, we have
\[
[\frac{\ln(1+z)}{\frac{z}{1+z}}<\frac{1+n^{1-s}}{2}]'_z>0.
\]
Therefore, it suffices to have 
\[
n^{1-s}>18\ln\frac{9}{8}-1 \coloneqq c_0>1.
\]
As a result, we have $m^{\star}\in(\frac{n}{c_0}, n)$, and 
\begin{align*}
C_1(m^{\star}, \varepsilon)&=\cO(\frac{n+m^{\star}}{\ln(1+\frac{\sqrt{m^{\star}}}{8\kappa_f})}\ln{\frac{1}{\varepsilon}})\\
&=\cO(\frac{n}{\frac{\sqrt{n}}{8\kappa_f}}\ln{\frac{1}{\varepsilon}})=\cO(\kappa_f\sqrt{n}\ln{\frac{1}{\varepsilon}}),
\end{align*}
where in the second equality we have used $\kappa_f>n^{\frac{1}{2}}$.

For our iPreSVRG in Algorithm \ref{alg: inexact preconditioned svrg}, we have
\[
C'_1(m, \varepsilon)=\cO(\frac{n+(1+pd)m}{\ln(1+\frac{1}{4}m\eta \sigma^M)}\ln{\frac{1}{\varepsilon}}),
\]
thanks to Lemma \ref{lem: choice of p with FISTA with restart}, $p$ can be chosen as
\[
p=\cO(\sqrt{\kappa(M)}\ln\big(\sqrt{\kappa^M_f}\kappa(M)\big),
\]
furthermore, we can take $\eta=\frac{1}{2\sqrt{m}L_f}$ due to Theorem \ref{thm: convergence of inexact preconditioned svrg}.

Under these settings, we have
\[
C'_1(m, \varepsilon)=\cO(\frac{n+(1+pd)m}{\ln(1+\frac{1}{8}\frac{\sqrt{m}}{\kappa^M_f})}\ln{\frac{1}{\varepsilon}}).
\]
Let us take $m=m'=\lceil\frac{n}{1+pd}\rceil$. 

If $n>1+pd$, or equivalently $\kappa_f<n^2d^{-2}$, then
\[
C'_1(m', \varepsilon)=\cO(\frac{n}{\ln(1+\frac{1}{8}\frac{\sqrt{n}}{\sqrt{pd}\kappa^M_f})}\ln{\frac{1}{\varepsilon}}).
\]
Since $p=\cO\bigg(\sqrt{\kappa(M)}\ln\big(\sqrt{\kappa^M_f}\kappa(M)\big)\bigg),$ we know that when $(\kappa^M_f)^2\sqrt{\kappa(M)}d < n$, or equivalently $\kappa_f<n^2d^{-2}$, we have
\[
\ln(1+\frac{1}{8}\frac{\sqrt{n}}{\sqrt{pd}\kappa^M_f})=\cO(\ln n), 
\]
therefore
\[
C'_1(m', \varepsilon)=\cO(n\ln{\frac{1}{\varepsilon}}),
\]
and
\[
\frac{\min_{m\geq 1}C_1'(m,\varepsilon)}{\min_{m\geq 1}C_1(m,\varepsilon)}\leq \frac{C'_1(m', \varepsilon)}{C_1(m^{\star}, \varepsilon)} =\cO(\frac{\sqrt{n}}{\kappa_f}).
\]


If $n\leq 1+pd$, or equivalently $\kappa_f>n^2d^{-2}$, then $m=1$ and
\[
C'_1(m, \varepsilon)=\cO(\frac{\sqrt{\kappa(M)}d}{\ln(1+\frac{1}{8}\frac{1}{\kappa^M_f})}\ln{\frac{1}{\varepsilon}}),
\]
therefore
\[
\frac{\min_{m\geq 1}C_1'(m,\varepsilon)}{\min_{m\geq 1}C_1(m,\varepsilon)}\leq \frac{C'_1(1, \varepsilon)}{C_1(m^{\star}, \varepsilon)} =\cO(\frac{\sqrt{\kappa(M)}d}{\kappa_f\sqrt{n}\ln(1+\frac{1}{8}\frac{1}{\kappa^M_f})}).
\]
Since $\kappa(M)\approx \kappa_f\gg \kappa^M_f$, this ratio becomes $\cO(\frac{d}{\sqrt{n\kappa_f}})$
\end{proof}

\begin{proof}[Proof of Theorem \ref{thm: inexact preconditioned Katyusha X vs Katyusha X}]
The proof of Theorem \ref{thm: inexact preconditioned Katyusha X vs Katyusha X} is similar and is omitted.
\end{proof}

\end{document}